\documentclass[11pt,a4paper]{amsart}
\usepackage{amssymb}
\usepackage{amscd}
\usepackage{amsmath,amsfonts,stmaryrd, graphicx}
\usepackage[usenames,dvipsnames,svgnames]{xcolor}

\usepackage{silence}
\usepackage[colorlinks, urlcolor=Navy, linkcolor=Navy, citecolor=Navy]{hyperref}
\usepackage{microtype}

\usepackage[abbrev,nobysame]{amsrefs}
\renewcommand{\MR}[1]{}

\usepackage{tikz}
\usepackage{tikz-cd}

\theoremstyle{plain}
\newtheorem{thm}{Theorem}[section]
\newtheorem{lem}[thm]{Lemma}

\newtheorem{pro}[thm]{Proposition}
\newtheorem{cor}[thm]{Corollary}
\newtheorem{thm*}{Theorem}
\theoremstyle{definition}
\newtheorem{dfn}[thm]{Definition}
\newtheorem{exa}[thm]{Example}
\newtheorem{rem}[thm]{Remark}

\numberwithin{equation}{section}

\DeclareMathOperator{\Hom}{Hom}
\DeclareMathOperator{\Aut}{Aut}
\DeclareMathOperator{\End}{End}
\DeclareMathOperator{\Ext}{Ext}

\DeclareMathOperator{\thick}{thick}

\DeclareMathOperator{\modu}{mod}

\newcommand{\lmod}[1]{\text{$#1$-$\modu$}}

\newcommand{\proje}[1]{\text{$#1$-$\mathrm{proj}$}}
\newcommand{\injec}[1]{\text{$#1$-$\mathrm{inj}$}}

\newcommand{\isoto}{\xrightarrow{\sim}}

\DeclareMathOperator{\id}{id}
\DeclareMathOperator{\idim}{idim}
\DeclareMathOperator{\pdim}{pdim}

\newcommand{\inj}{\hookrightarrow}

\newcommand{\oTo}{\xymatrix{ \ar@{^{(}->}[r]|{\mathbf{O}}& }} 
\newcommand{\cTo}{\xymatrix{ \ar@{^{(}->}[r]|{\mathbf{|}}& }} 
\newcommand{\coTo}{\xymatrix{ \ar@{^{(}->}[r]|{\mathbf{O}}|{\mathbf{|}}& }} 
\newcommand{\surj}{\twoheadrightarrow }

\newcommand{\op}{\mathrm{op}}

\DeclareMathOperator{\Ker}{ker}
\DeclareMathOperator{\coKer}{coker}
\DeclareMathOperator{\Bild}{im}

\DeclareMathOperator{\add}{add}
\DeclareMathOperator{\gldim}{gldim}
\DeclareMathOperator{\domdim}{domdim}
\newcommand{\kdual}{\mathrm{D}}

\newcommand{\D}{\mathbb{D}}

\newcommand{\K}{\mathbb{K}}

\newcommand{\Z}{\mathbb{Z}}

\newcommand{\mcA}{\mathcal{A}}

\newcommand{\mcC}{\mathcal{C}}
\newcommand{\mcD}{\mathcal{D}}

\newcommand{\mcH}{\mathcal{H}}

\newcommand{\mcT}{\mathcal{T}}

\newcommand{\mcX}{\mathcal{X}}
\newcommand{\mcY}{\mathcal{Y}}
\newcommand{\mcZ}{\mathcal{Z}}

\DeclareMathOperator{\gen}{gen}
\DeclareMathOperator{\cogen}{cogen}

\newcommand{\homot}[2][]{\mathcal{K}^{#1}(#2)}
\newcommand{\dercat}[2][]{\mathcal{D}^{#1}(#2)}
\newcommand{\bound}{\mathrm{b}}

\usepackage{todonotes}

\DeclareMathOperator{\TTF}{TTF}

\renewcommand{\epsilon}{\varepsilon}

\newcommand{\shiftmod}[1]{T_{#1}}
\newcommand{\coshiftmod}[1]{C^{#1}}
\newcommand{\shiftalg}[1]{B_{#1}}
\newcommand{\coshiftalg}[1]{B^{#1}}

\begin{document}

\title[Special tilting modules for algebras with positive dominant dimension]{Special tilting modules\\for algebras with positive dominant dimension}

\author{Matthew Pressland}
\address{Matthew Pressland\\Institut f\"ur Algebra und Zahlentheorie\\Universit\"at Stuttgart\\Pfaffenwaldring~57\\70569 Stuttgart\\Germany}
\email{presslmw@mathematik.uni-stuttgart.de}

\author{Julia Sauter}
\address{Julia Sauter\\Fakult\"at f\"ur Mathematik\\Universit\"at Bielefeld\\Postfach 100 131\\D-33501 Bielefeld\\Germany}
\email{jsauter@math.uni-bielefeld.de}


\date{\today}
\dedicatory{Dedicated to Idun Reiten on the occasion of her 75\textsuperscript{th} birthday.}

\begin{abstract}
We study certain special tilting and cotilting modules for an algebra with positive dominant dimension, each of which is generated or cogenerated (and usually both) by projective-injectives. These modules have various interesting properties, for example that their endomorphism algebras always have global dimension at most that of the original algebra. We characterise minimal $d$-Auslander--Gorenstein algebras and $d$-Auslander algebras via the property that these special tilting and cotilting modules coincide. By the Morita--Tachikawa correspondence, any algebra of dominant dimension at least $2$ may be expressed (essentially uniquely) as the endomorphism algebra of a generator-cogenerator for another algebra, and we also study our special tilting and cotilting modules from this point of view, via the theory of recollements and intermediate extension functors.
\end{abstract}
\maketitle


\section{Introduction}

In \cite{CBSa}, Crawley-Boevey and the second author associated to each Auslander algebra a distinguished tilting-cotilting module $T$, with the property that it is both generated and cogenerated by a projective-injective module. In this paper, we study more general instances of tilting modules generated by projective-injectives, and cotilting modules cogenerated by projective-injectives. In contrast to the case of Auslander algebras, we consider here tilting and cotilting modules of arbitrary finite projective or injective dimension.

More precisely, let $\Gamma$ be a finite-dimensional algebra with dominant dimension $d+1$ (see Definition~\ref{dom-dim}). Then for every $1\leq k\leq d$, we may specify a specific `shifted' $k$-tilting module $\shiftmod{k}$ and `coshifted' $k$-cotilting module $\coshiftmod{k}$, 
which are both generated and cogenerated by projective-injectives.
We are also interested in the resulting shifted and coshifted algebras $\shiftalg{k}=\End_\Gamma(\shiftmod{k})^{\op}$ and $\coshiftalg{k}=\End_\Gamma(\coshiftmod{k})^{\op}$. Some of our results will also apply to the degenerate cases of $k=0$ and $k=d+1$.

Despite the relatively simple construction of these modules and algebras, they appear not to have been studied in much detail, particularly for $k>1$, until very recently. Indeed, the only general result we are aware of prior to our work is by Chen and Xi \cite{CX}, who call the $\shiftmod{k}$ `canonical tilting modules' and obtain results on the dominant dimension of the shifted algebras $\shiftalg{k}$.
Some of our results on the single tilting module $\shiftmod{1}$, notably those of Section~\ref{dAG-algebras} on minimal $1$-Auslander--Gorenstein algebras, were obtained independently by Nguyen, Reiten, Todorov and Zhu \cite{NRTZ}.

The primary purpose of this paper is to collect information about the modules $\shiftmod{k}$ relevant to tilting theory; for example, by computing various subcategories determined by the tilting module and used in Miyashita's generalisation of the Brenner--Butler theorem, and by studying recollements involving the tilted algebras $\shiftalg{k}$. We take our lead from the questions answered in \cite{CBSa} for $\shiftmod{1}$ in the case that $\Gamma$ is an Auslander algebra. These considerations are applied in \cite{CBSa} to construct desingularisations for varieties of modules, but the geometric arguments depend crucially on working in a low homological dimension. Nevertheless, we are able to extend most of the homological results of \cite{CBSa} to our much higher level of generality---in some cases, also simplifying the arguments---and we expect these properties to be of independent interest. We note that the geometric statements of \cite{CBSa} can also be extended by generalising in a somewhat different direction, as we explain in \cite{PS2}.

The non-degenerate shifted and coshifted modules are defined when $\Gamma$ has dominant dimension at least $2$, and in this case $\Gamma\cong\End_A(E)^{\op}$ for a generating-cogenerating module $E$ over a finite-dimensional algebra $A$. In fact, assuming for simplicity that all objects are basic, the assignment $(A,E)\mapsto\End_A(E)^{\op}$ induces a bijection
\[\{[A,E]:\text{$E$ a generating-cogenerating $A$-module}\}\isoto\{[\Gamma]:\domdim{\Gamma}\geq2\},\]
with objects considered up to isomorphism on each side\footnote{We say that $(A,E)$ and $(A',E')$ are isomorphic if there is an isomorphism $\varphi\colon A\isoto A'$ such that $\varphi^*E'\cong E$. The reader is warned that $(A,E)\cong(A,E')$ does not imply that $E\cong E'$ as $A$-modules, but only that $E\cong\varphi^*E'$ for some $\varphi\in\Aut(A)$.} \cites{Tach-PJM,Morita71}, a result sometimes called \cites{RingelNotes, GK} the \emph{Morita--Tachikawa correspondence}. In this context, it will often be convenient for us to express results on the special cotilting $\Gamma$-modules in terms of the pair $(A,E)$ on the other side of the correspondence, and as such the following definition will be convenient throughout the paper.

\begin{dfn}
\label{MT-triple}
A \emph{Morita--Tachikawa triple} $(A,E,\Gamma)$ consists of a finite-dimensional algebra $A$, a generating-cogenerating $A$-module $E$, and $\Gamma\cong\End_A(E)^{\op}$.
\end{dfn}

Thus, the set of isoclasses of Morita--Tachikawa triples in which all three objects are basic is the graph of the Morita--Tachikawa correspondence. Given a basic algebra $\Gamma$ of dominant dimension at least $2$, it appears in the (unique up to isomorphism) Morita--Tachikawa triple
\[(A=\End_\Gamma(\Pi)^{\op},E=\kdual\Pi,\Gamma),\]
for $\Pi$ a maximal projective-injective summand of $\Gamma$, and $\kdual$ the usual duality over the base field. 

The structure of the paper is as follows. We give the definitions and preliminary results in Section~\ref{shifted-algebras}, including the observation (Proposition~\ref{gldim-bound}) that $\gldim B \leq \gldim \Gamma$ whenever $B$ is one of the shifted or coshifted algebras of $\Gamma$, and a characterisation (Proposition~\ref{shifted-char}) of the algebras $B$ arising in this way. In Section~\ref{dAG-algebras}, we investigate the modules $\shiftmod{k}$ and $\coshiftmod{k}$ in the context of higher Auslander--Reiten theory, which provides a wealth of examples of algebras with high dominant dimension. The main result of this section is the following.

\begin{thm*}[Theorem~\ref{tilt-cotilt-dAG}]
Let $\Gamma$ be a finite-dimensional non-selfinjective algebra with $\domdim{\Gamma}=d+1$. The following are equivalent:
\begin{itemize}
\item[(i)]$\Gamma$ is a minimal $d$-Auslander--Gorenstein algebra,
\item[(ii)]$\shiftmod{k}=\coshiftmod{d+1-k}$ for all $0\leq k\leq d+1$, and
\item[(iii)]there exists $M\in\lmod{\Gamma}$ that is both a shifted and a coshifted module.
\end{itemize}
Under these conditions, $\Gamma$ is a $d$-Auslander algebra if and only if $\gldim{\Gamma}<\infty$.
\end{thm*}
The definition of a minimal $d$-Auslander--Gorenstein algebra, due to Iyama and Solberg \cite{IS}, may be found below (Definition~\ref{dAG}). This result generalises \cite[Lem.~1.1]{CBSa} for ($1$-)Auslander algebras, and also \cite[Thm.~2.4.12]{NRTZ} by Nguyen, Reiten, Todorov and Zhu, who proved it independently in the case $d=1$.

If $\Pi$ is the maximal projective-injective summand of $\Gamma$, it is a summand of every tilting or cotilting $\Gamma$-module. Thus if $B$ is the endomorphism algebra of such a module, it has an idempotent given by projection onto $\Pi$, yielding a recollement involving the categories $\lmod{B}$ and $\lmod{\End_\Gamma(\Pi)^{\op}}$; note that if $\domdim{\Gamma}\geq2$ then $\End_\Gamma(\Pi)^{\op}$ is the algebra $A$ from the Morita--Tachikawa triple involving $\Gamma$. In Section~\ref{homotopy-cats}, we study these recollements for the shifted and coshifted algebras. In particular, we give in Theorems~\ref{cdescription} and \ref{cdescription*} an explicit formula for the intermediate extension functor in such a recollement; this functor is, by definition, the image of the universal map from the restriction functor's left adjoint to its right adjoint.

To obtain this formula, we show that each shifted and coshifted algebra of $\Gamma$ can be described in terms of its Morita--Tachikawa partner $(A,E)$, as in the following theorem.

\begin{thm*}[Theorem~\ref{HtoB}]
Let $(A,E,\Gamma)$ be a Morita--Tachikawa triple. Then for all $0\leq k\leq\domdim{\Gamma}$, there are isomorphisms
\[\shiftalg{k}\cong\End_{\homot[\bound]{A}}(E_k)^{\op},\qquad\coshiftalg{k}\cong\End_{\homot[\bound]{A}}(E^k)^{\op},\]
where $\shiftalg{k}$ and $\coshiftalg{k}$ are the shifted and coshifted algebras of $\Gamma$, and $E_k$ and $E^k$ are certain bounded complexes of $A$-modules, defined explicitly in Theorem~\ref{HtoB}.
\end{thm*}

This result generalises \cite[Prop.~5.5]{CBSa} for the case that $\Gamma$ is an Auslander algebra and $k=1$. The proof we give here is different, and more conceptual.



A $k$-tilting or $k$-cotilting $\Gamma$-module with endomorphism algebra $B$ defines $k+1$ pairs of equivalent subcategories in $\lmod{\Gamma}$ and $\lmod{B}$; in the classical case $k=1$, the two subcategories on each side form a torsion pair. In Section~\ref{subcategories}, we describe these subcategories for the shifted and coshifted modules, often in terms of generation or cogeneration by certain projective or injective modules.

In Section~\ref{int-exts}, we consider again the recollements involving $\lmod{B}$ and $\lmod{A}$, where $B$ is one of the shifted or coshifted algebras of an algebra $\Gamma$ in a Morita--Tachikawa triple $(A,E,\Gamma)$. Recall from general tilting theory that $\shiftalg{k}$, as a tilt of $\Gamma$ by $\shiftmod{k}$, has a preferred cotilting module $\kdual\shiftmod{k}$. Similarly $\coshiftalg{k}$ has the preferred tilting module $\kdual\coshiftmod{k}$. We prove the following, again generalising the results of \cite{CBSa} for the case that $\Gamma$ is the Auslander algebra of $A$.

\begin{thm*}[Theorems~\ref{shifted-intexts}, \ref{coshifted-intexts}]
Let $(A,E,\Gamma)$ be a Morita--Tachikawa triple and $0<k<\domdim{\Gamma}$. Denoting by $c_k$ and $c^k$ the intermediate extension functors in the recollements relating $\lmod{\shiftalg{k}}$ and $\lmod{\coshiftalg{k}}$ respectively with $\lmod{A}$, we have
\[c_k(E)=\kdual\shiftmod{k},\qquad c^k(E)=\kdual\coshiftmod{k}.\]
\end{thm*}


Throughout the paper, all algebras are finite-dimensional $\K$-algebras over some field $\K$, and, without additional qualification, `module' is taken to mean `finitely-generated left module'. As mentioned above, $\kdual=\Hom_{\K}(-,\K)$ denotes the $\K$-linear dual. Morphisms are composed right-to-left.

\section{Shifted modules and algebras}
\label{shifted-algebras}

Throughout this section, we fix a finite-dimensional algebra $\Gamma$, assumed for simplicity to be basic, over a field $\K$. The goal of this section is to define certain special tilting and cotilting $\Gamma$-modules in the case that $\Gamma$ has positive dominant dimension, and list some of their basic properties.

\begin{dfn}
\label{dom-dim}
Let $k$ be a non-negative integer. We say that $\Gamma $ has \emph{dominant dimension} at least $k$, and write $\domdim \Gamma \geq k$, if the regular module ${}_\Gamma\Gamma$ has an injective resolution
\[\begin{tikzcd}[column sep=20pt]
0\arrow{r}&\Gamma\arrow{r}&\Pi_0\arrow{r}&\cdots\arrow{r}&\Pi_{k-1}\arrow{r}&\cdots
\end{tikzcd}\]
with $\Pi_0, \ldots , \Pi_{k-1}$ projective-injective; when $k=0$, this condition is taken to be empty.  As the notation suggests, we write $\domdim\Gamma=k$ if $\domdim\Gamma\geq k$ and $\domdim\Gamma\not\geq k+1$.
\end{dfn}

\begin{rem} \label{left-right} As always, we refer to left $\Gamma$-modules in our definition of dominant dimension. However, the analogous definition using right modules is equivalent to ours by a result of M\"{u}ller \cite[Thm.~4]{Mue}. As a consequence, $\domdim\Gamma\geq k$ if and only if $\kdual \Gamma $ has a projective resolution 
\[\begin{tikzcd}[column sep=20pt]
\cdots\arrow{r}&\Pi^{k-1}\arrow{r}&\cdots\arrow{r}&\Pi^0\arrow{r}&\kdual\Gamma\arrow{r}&0
\end{tikzcd}\]
with $\Pi^0, \ldots , \Pi^{k-1}$ projective-injective.
\end{rem}

\begin{dfn} \label{dfn-tilting}
Let $k\geq 0$. We say that $T\in\lmod{\Gamma}$ is a \emph{$k$-tilting} module if
\begin{itemize}
\item[(T1)] $\pdim T \leq k$,
\item[(T2)] $\Ext^i_\Gamma(T,T)=0$ for $i>0$, and
\item[(T3)] there is an $\add{T}$-coresolution of $\Gamma$ of length $k$, i.e.\ an exact sequence
\[\begin{tikzcd}[column sep=20pt]
0\arrow{r}&\Gamma\arrow{r}&t_0\arrow{r}&\cdots\arrow{r}&t_k\arrow{r}&0
\end{tikzcd}\]
with $t_j \in \add T$ for $0\leq j\leq k$.
\end{itemize}
We say a $k$-tilting module $T$ is \emph{$P$-special} for a projective module $P$ if $P\in\add{T}$ and there is a sequence as in (T3) with $t_j\in \add P$ for $0\leq j \leq k-1$.

Dually, we say that $C$ is a \emph{$k$-cotilting module} if 
\begin{itemize}
\item[(C1)] $\idim C \leq k$,
\item[(C2)] $\Ext^i_\Gamma(C,C) =0$ for $i>0$, and
\item[(C3)] there is an $\add{C}$-resolution of $\kdual\Gamma$ of length $k$, i.e.\ an exact sequence
\[\begin{tikzcd}[column sep=20pt]
0\arrow{r}&c^k\arrow{r}&\cdots\arrow{r}&c^0\arrow{r}&\kdual\Gamma\arrow{r}&0
\end{tikzcd}\]
with $c^j\in\add{C}$ for $0\leq j\leq k$.
\end{itemize}
We say a $k$-cotilting module $C$ is \emph{$I$-special} for an injective module $I$ if $I\in\add{C}$ and there is a sequence as in (C3) with $c^j \in \add I$ for $0\leq j \leq k-1$.
\end{dfn}

\begin{pro}
\label{shifting}
Let $\Gamma$ be a finite-dimensional algebra, and let $\Pi$ be a maximal projective-injective summand of $\Gamma$. Then there exists a basic $\Pi$-special $k$-tilting $\Gamma$-module and a basic $\Pi$-special $k$-cotilting $\Gamma$-module $\coshiftmod{k}$ if and only if $\domdim{\Gamma}\geq k$. These modules are unique up to isomorphism.
\end{pro}

\begin{proof}
We prove the statements involving $\shiftmod{k}$, those for $\coshiftmod{k}$ being dual. If $\Gamma$ has a $\Pi$-special $k$-tilting module, then $\domdim{\Gamma}\geq k$ by (T3). Conversely, if $\domdim{\Gamma}\geq k$, there is an exact sequence
\begin{equation}
\label{Tk-def}
\begin{tikzcd}[column sep=20pt]
0\arrow{r}&\Gamma\arrow{r}&\Pi_0\arrow{r}&\cdots\arrow{r}&\Pi_{k-1}\arrow{r}&T\arrow{r}&0
\end{tikzcd}
\end{equation}
with $\Pi_i$ projective-injective for $0\leq i\leq k-1$. Let $\shiftmod{k}$ be a basic module with $\add{\shiftmod{k}}=\add(T\oplus\Pi)$. Then $\shiftmod{k}$ satisfies (T1) and (T3), and is $\Pi$-special, by \eqref{Tk-def}. A standard homological argument, involving the application of the functors $\Hom_\Gamma(T_k,-)$ and $\Hom_\Gamma(-,T_k)$ to the short exact sequences coming from \eqref{Tk-def}, shows that $\Ext^i_\Gamma(T_k,T_k)=\Ext^i_\Gamma(\Gamma,\Gamma)=0$ for $i>0$, so $T_k$ satisfies (T2).

Any two $\Pi$-special $k$-tilting $\Gamma$-modules are, by definition, $k$-th cosyzygies of the regular module $\Gamma$. Thus if $T'$ is an arbitrary $k$-th cosyzygy of $\Gamma$, it differs from $\shiftmod{k}$ only by the possible removal of projective-injective summands and addition of injective summands, so $T\in\add{T'}$, where $T$ is as in \eqref{Tk-def}. If $T'$ is tilting then we must also have $\Pi\in\add{T'}$, so $\shiftmod{k}\in\add{T'}$. If $T'$ is basic, it then follows that $T'\cong\shiftmod{k}$ since all tilting modules have the same number of indecomposable summands up to isomorphism.
\end{proof}

\begin{dfn}
For $\Gamma$ a finite-dimensional algebra with $\domdim{\Gamma}\geq k$, write $\shiftmod{k}$ and $\coshiftmod{k}$ for basic $\Pi$-special $k$-tilting and $k$-cotilting modules respectively, these modules being unique up to isomorphism by Proposition~\ref{shifting}. We call $\shiftmod{k}$ the \emph{$k$-shifted} module of $\Gamma$, and $\coshiftmod{k}$ the the \emph{$k$-coshifted} module of $\Gamma$. The algebras
\[
\shiftalg{k} =\End_\Gamma (\shiftmod{k})^{\op}, \quad \coshiftalg{k} =\End_\Gamma (\coshiftmod{k})^{\op}, 
\]
are called respectively the $k$-shifted and $k$-coshifted algebras of $\Gamma$.
\end{dfn}

\begin{rem}\label{opposite}
If $\domdim \Gamma \geq k$, then $\domdim \Gamma^{\op} \geq k$ by Remark~\ref{left-right}. The dual of the $k$-coshifted $\Gamma^{\op}$-module is the $k$-shifted $\Gamma$-module.
\end{rem}

The modules $\shiftmod{k}$ appeared briefly as an example in a paper of Chen and Xi \cite{CX}, where they are called `canonical tilting modules'. It is well-known that if $T$ is a $k$-tilting $\Gamma$-module with $B=\End_\Gamma(T)^{\op}$, then the right derived functor of $\Hom_\Gamma(T,-)$ and the left derived functor of $\kdual \Hom_B(-, \kdual T)$ are quasi-inverse triangle equivalences between the bounded derived categories $\dercat[\bound]{\Gamma}$ and $\mcD^b(B)$, cf.\ \cite[Thm.~2.1]{CPS}. In particular, $\Gamma$ is derived equivalent to all of its $k$-shifted and $k$-coshifted algebras.

The proof of Proposition~\ref{shifting} illustrates that the shifted and coshifted modules are related to $\Gamma$ and $\kdual\Gamma$ analogously to the way in which an arbitrary module over a selfinjective algebra is related to its shifts in the stable module category (hence our choice of terminology). Despite this analogy, the case in which $\Gamma$ is selfinjective does not provide any interesting examples of our constructions, since in this case $\shiftmod{k}\cong\Gamma\cong\coshiftmod{k}$ for all $k\geq0$---indeed, there are no other tilting or cotilting $\Gamma$ modules. More interestingly, any non-selfinjective algebra displays very different behaviour, with no coincidences between any two of the shifted modules. This follows from the following observation.

\begin{pro}
\label{shifted-projdim}
If $\Gamma$ is not selfinjective, then $\pdim{\shiftmod{k}}=k$ (and dually $\idim{\coshiftmod{k}}=k$) for all $0\leq k\leq\domdim\Gamma$.
\end{pro}
\begin{proof}
Let $\shiftmod{k}^\circ$ be the maximal non-projective-injective summand of $\shiftmod{k}$
, and let $P$ be the maximal non-injective summand of $\Gamma$, which is non-zero by assumption. As in the proof of Proposition~\ref{shifting}, taking the minimal injective resolution of $P$ and truncating yields an exact sequence
\[\begin{tikzcd}[column sep=20pt]
0\arrow{r}&P\arrow{r}&\Pi_0\arrow{r}&\Pi_1\arrow{r}&\cdots\arrow{r}&\Pi_{k-1}\arrow{r}&\shiftmod{k}^\circ\arrow{r}&0
\end{tikzcd}\]
with $\Pi_j\in\add{\Pi}$ projective for all $j$, so this sequence is a minimal projective resolution of $\shiftmod{k}^\circ$. Since $\shiftmod{k}=\shiftmod{k}^\circ\oplus\Pi$ with $\Pi$ projective, we conclude that $\pdim{\shiftmod{k}}=k$. 
\end{proof}

\begin{rem}
It follows from Proposition~\ref{shifted-projdim} that any counterexample to the Nakayama conjecture (i.e.\ a non-selfinjective algebra of infinite dominant dimension) would have a tilting module of each possible projective dimension. This observation shows that the truth of the finitistic dimension conjecture (for a family of algebras) implies the truth of the Nakayama conjecture (for the same family), and is essentially equivalent to Tachikawa's proof of this fact \cite[\S8]{T}.
\end{rem}

To give a slightly different characterisation of the modules $T_k$ and $C^k$, we introduce the following definitions, which will also be useful in Section~\ref{subcategories}.

\begin{dfn}
\label{gen-cogen}
Let $\mcA$ be an abelian category, and $X\in\mcA$ an object. For $k\geq0$, define $\gen_k(X)$ to be the full subcategory of $\mcA$ on objects $M$ such that there exists an exact sequence
\[\begin{tikzcd}[column sep=20pt]
X^k\arrow{r}&\dotsb\arrow{r}&X^0\arrow{r}&M\arrow{r}&0
\end{tikzcd}\]
with $X^i\in\add{X}$ for $0\leq i\leq k$, remaining exact under the functor $\Hom_{\mcA}(X,-)$. Dually, $\cogen^k(X)$ is the full subcategory of $\mcA$ on objects $N$ such that there exists an exact sequence
\[\begin{tikzcd}[column sep=20pt]
0\arrow{r}&N\arrow{r}&X_0\arrow{r}&\dotsb\arrow{r}&X_k
\end{tikzcd}\]
with $X_i\in\add{X}$ for all $0\leq i\leq k$, remaining exact under the functor $\Hom_{\mcA} (-,X)$. Note that the conditions involving the Hom-functor are automatic when $X$ is projective or injective respectively, or when $k\leq 1$. When $k=0$, we omit it from the notation and refer simply to $\gen(X)$ and $\cogen(X)$.
It is both natural and convenient to define $\gen_{-1}(X)=\mcA=\cogen^{-1}(X)$.
\end{dfn}

\begin{pro}
\label{gen-cogen-characterisation}
Let $\Pi$ be a maximal projective-injective summand of $\Gamma$, and $k\geq0$.
\begin{itemize}
\item[(a)]The subcategory $\gen_{k-1}(\Pi)\subseteq\lmod{\Gamma}$ contains a $k$-tilting object if and only if $\domdim{\Gamma}\geq k$. Any basic such $k$-tilting object is isomorphic to the $\Pi$-special $k$-tilting module $\shiftmod{k}$.
\item[(b)]The subcategory $\cogen^{k-1}(\Pi)\subseteq\lmod{\Gamma}$ contains a $k$-cotilting object if and only if $\domdim{\Gamma}\geq k$. Any basic such $k$-cotilting object is isomorphic to the $\Pi$-special $k$-cotilting module $\coshiftmod{k}$.
\end{itemize}
\end{pro}
\begin{proof}
We prove only (a), since (b) is dual. If $\domdim{\Gamma}\geq k$, then the module $\shiftmod{k}$ from Proposition~\ref{shifting} lies in $\gen_{k-1}(\Pi)$. Conversely, if $T\in\gen_{k-1}(\Pi)$ is $k$-tilting, it has projective dimension at most $k$, and the minimal projective resolution of $T$ is of the form
\[\begin{tikzcd}[column sep=20pt]
0\arrow{r}&P\arrow{r}&\Pi_{k-1}\arrow{r}&\cdots\arrow{r}&\Pi_0\arrow{r}&T\arrow{r}&0
\end{tikzcd}\]
for $\Pi_i\in\add{\Pi}$ and $P$ projective. Without loss of generality, we may assume $T$, like $\Gamma$, is basic. Then the number of indecomposable summands of $P$ is the number of non-projective-injective summands of $T$, which is the number of non-projective-injective summands of $\Gamma$. Thus there is an exact sequence
\[\begin{tikzcd}[column sep=20pt]
0\arrow{r}&\Gamma\arrow{r}&\Pi_{k-1}\oplus\Pi\arrow{r}&\cdots\arrow{r}&\Pi_0\arrow{r}&T\arrow{r}&0,
\end{tikzcd}\]
from which it follows simultaneously that $\domdim{\Gamma}\geq k$ and that $T$ is $\Pi$-special, hence isomorphic to $T_k$ by Proposition~\ref{shifting}.
\end{proof}



It is possible to identify those algebras that may be obtained as $k$-shifted or $k$-coshifted algebras intrinsically, via the existence of cotilting or tilting modules with special properties. As usual, we write $\nu=\kdual\Hom_A(-,A)$ and $\nu^-=\Hom_A(\kdual A,-)$ for the Nakayama functors of an algebra $A$.

\begin{lem}\label{proj-injAndNakayama}Let $T$ be a $k$-tilting $\Gamma$-module with endomorphism algebra $B$. By the Brenner--Butler tilting theorem \cite{BrenBut}, $C=\kdual T$ is a $k$-cotilting $B$-module with endomorphism algebra $\Gamma$.
\begin{itemize}
\item[(1)]If $T$ is $P$-special for a projective $\Gamma$-module $P$, then $C$ is $I_P$-special for $I_P=\kdual \Hom_\Gamma (P, T)$. Dually, if $C$ is $I$-special for an injective $B$-module $I$, then $T$ is $P^I$-special for $P^I= \Hom_B(C,I)$.
\item[(2)] Let $\Pi\in \add T $ be projective-injective. Then the projective $B$-module
$P_\Pi= \Hom_{\Gamma }(T, \Pi )$ and the injective $B$-module $I_\Pi= \kdual \Hom_{\Gamma} (\Pi, T)$ satisfy $I_\Pi=\nu P_\Pi$. Dually, if $\Pi \in \add C$ is projective-injective, then the $\Gamma$-modules
$P^\Pi= \Hom_{B}(C, \Pi )$ and $I^\Pi= \kdual \Hom_{B} (\Pi, C)$ satisfy $I^\Pi=\nu P^\Pi$.
\item[(3)] If $P$ is a projective $\Gamma$-module with $P, \nu P \in \add T$, then $I_P:= \kdual\Hom_{\Gamma} (P,T)$ is a projective-injective $B$-module. Dually, if $I$ is an injective $B$-module with $I, \nu^- I \in \add C$, then 
$P^I:= \Hom_{B }(C, I)$ is a projective-injective $\Gamma $-module.
\end{itemize}
\end{lem}

\begin{proof}
As usual, we give the proof only for the first item in each pair of dual statements.
\begin{itemize}
\item[(1)] This follows by applying $\kdual\Hom_\Gamma(-,T)$ to the exact sequence from (T3), using that $T$ is $P$-special.
\item[(2)]Since $\Hom_\Gamma(T,-)\colon\add{T}\to\proje{B}$ is fully faithful, we have
\[\nu P=\kdual\Hom_B(\Hom_\Gamma(T,\Pi),\Hom_\Gamma(T,T))=\kdual\Hom_\Gamma(\Pi,T)=I.\]
\item[(3)]Since $\nu P\in\add{T}$, the module $\Hom_{\Gamma}(T,\nu P)$ is projective. Since $P\in\add{T}$, the Nakayama formula implies that $\Hom_{\Gamma}(T,\nu P)\cong\kdual\Hom_{\Gamma}(P,T)$ is also injective. \qedhere
\end{itemize}

\end{proof}

\begin{pro}
\label{shifted-char}
 A finite-dimensional basic algebra $B$ is isomorphic to a $k$-shifted algebra if and only if there is an injective $B$-module $I$ and an $I$-special $k$-cotilting $B$-module $C$ with $\nu^- I \in \add C$. 
Under this isomorphism, $C$ is the dual of the $k$-shifted module. 

Dually, a finite-dimensional basic algebra $B$ is isomorphic to a $k$-coshifted algebra if and only if 
there exists a projective $B$-module $P$ and a $P$-special $k$-tilting $B$-module $T$ with $\nu P \in \add T$.  Under this isomorphism, $T$ is the dual of the $k$-coshifted module. 
\end{pro}
\begin{proof}
Let $\shiftmod{k}$ be the $k$-shifted module of an algebra $\Gamma$ with maximal projective-injective summand $\Pi$. Then by Lemma~\ref{proj-injAndNakayama}(1), $\kdual\shiftmod{k}$ is an $I_\Pi$-special $k$-cotilting $\shiftalg{k}$-module, where $I_\Pi=\kdual\Hom_\Gamma(\Pi,T)$. By Lemma~\ref{proj-injAndNakayama}(2), $\nu^-I_\Pi=\Hom_\Gamma(\shiftmod{k},\Pi)$ lies in $\add{\kdual T}$, since $\Pi\in\add{\kdual\Gamma}$.

Conversely, assume $B$, $C$ and $I$ are as in the statement, replacing $C$ and $I$ by basic modules with the same additive closure if necessary. Then $\Gamma=\End_B(C)^{\op}$ has a basic $k$-tilting module $T=\kdual C$, which is $P^I=\Hom_B(C,I)$-special by Lemma~\ref{proj-injAndNakayama}(1). By Lemma~\ref{proj-injAndNakayama}(3), $P^I$ is projective-injective. If $\Pi$ is the maximal projective-injective summand of $\Gamma$, then $\Pi$ is a summand of $T$ since $T$ is $k$-tilting, so $\Pi\in\gen(P^I)$ since $T$ is $P^I$-special. It follows that $\add{P^I}=\add{\Pi}$, and so $T\cong\shiftmod{k}$ is the $k$-shifted module of $\Gamma$ by Proposition~\ref{shifting}.

The second statement is proved dually, reversing the roles of $\Gamma$ and $B$ in Lemma~\ref{proj-injAndNakayama}.
\end{proof}


To close this section, we observe that if $\shiftalg{k}$ is the $k$-shifted algebra of $\Gamma$, then $\gldim{\shiftalg{k}}\leq\gldim{\Gamma}$, thus obtaining a tighter bound on this global dimension than is possible for endomorphism algebras of arbitrary tilting $\Gamma$-modules.

\begin{pro}
\label{gldim-bound}
Assume $\domdim{\Gamma}=d$, let $0\leq k\leq d$, and let $\shiftalg{k}$ and $\coshiftalg{k}$ be the $k$-shifted and $k$-coshifted algebras of $\Gamma$. Then
\begin{gather*}
\gldim{\Gamma}-k\leq\gldim{\shiftalg{k}}\leq\gldim{\Gamma},\\
\gldim{\Gamma}-k\leq\gldim{\coshiftalg{k}}\leq\gldim{\Gamma}.
\end{gather*}
\end{pro}
\begin{proof}
Write $n=\gldim{\Gamma}$, which without loss of generality we may assume to be finite. 
Since $\shiftmod{k}$ is $k$-tilting, it is well-known (see, for example, \cite[Prop.~III.3.4]{Hap}) that
\[n-k\leq\gldim{\shiftalg{k}}.\]
Since $\shiftmod{k}$ is a $k$-th cosyzygy of $\Gamma$ and $\idim{\Gamma}\leq n$, it follows that $\idim{T_k}\leq n-k$. By a result of Gastaminza, Happel, Platzeck, Redondo and Unger \cite[Prop.~2.1]{GHPRU}, we have
\[\gldim{\shiftalg{k}}\leq \pdim{\shiftmod{k}}+\idim{\shiftmod{k}}\leq k+n-k=n.\]
The second pair of inequalities is proved dually, using that $\coshiftmod{k}$ is $k$-cotilting with $\pdim{\coshiftmod{k}}\leq n-k$.
\end{proof}

In particular, this means that either $\gldim{\shiftalg{1}}=\gldim{\Gamma}$ or $\gldim{\shiftalg{1}}=\gldim{\Gamma}-1$. Nguyen, Reiten, Todorov and Zhu have shown that the latter holds if and only if $\pdim(\tau\shiftmod{1})<\gldim{\Gamma}$ \cite[Thm.~3.2.9]{NRTZ}. By \cite[Thm.~3.2]{GHPRU}, this is equivalent to the property that $\Ext^1_\Gamma(\tau\shiftmod{1},\shiftmod{1})=0$.


\section{\texorpdfstring{Shifting and coshifting for minimal $d$-Auslander--Gorenstein algebras}{Shifting and coshifting for minimal d-Auslander--Gorenstein algebras}}
\label{dAG-algebras}

In \cite{CBSa}, Crawley-Boevey and the second author considered the $1$-shifted and $1$-coshifted modules of an Auslander algebra, and noted that these two modules in fact coincide. In this section, we extend this result by showing that the families of shifted and coshifted modules of a general algebra $\Gamma$ with $\domdim{\Gamma}\geq2$ intersect if and only if $\Gamma$ is a minimal $d$-Auslander--Gorenstein algebra, as defined by Iyama and Solberg \cite{IS}, and in this case they even coincide completely.


\begin{dfn}
\label{dAG}
Let $\Gamma$ be a finite-dimensional $\K$-algebra, and let $d\geq1$. We say $\Gamma$ is \emph{$d$-Auslander--Gorenstein} if
\[\id\Gamma\leq d+1\leq\domdim\Gamma,\]
and that it is a \emph{$d$-Auslander algebra} if
\[\gldim\Gamma\leq d+1\leq\domdim\Gamma.\]
\end{dfn}

The definition of a $d$-Auslander algebra is due to Iyama \cite{IyAR} (see also \cite[Defn.~4.1]{IyAC} for more general versions), generalising Auslander for $d=1$ \cite{AQueenMary}.

Note that any $d$-Auslander algebra is minimal $d$-Auslander--Gorenstein, and a minimal $d$-Auslander--Gorenstein algebra is a $d$-Auslander algebra if and only if it has finite global dimension \cite[Prop.~4.8]{IS}. A selfinjective algebra is minimal $d$-Auslander--Gorenstein for all $d$, and so is a $d$-Auslander algebra for all $d$ if and only if it is semisimple. On the other hand, by \cite[Prop.~4.1]{IS}, any minimal $d$-Auslander--Gorenstein algebra $\Gamma$ that is not selfinjective satisfies $\id\Gamma=d+1=\domdim\Gamma$, so $d$ is uniquely determined. Similarly, any $d$-Auslander algebra $\Gamma$ that is not semisimple has $\gldim{\Gamma}=d+1=\domdim{\Gamma}$.

These classes of algebras are also interesting from the point of view of the Morita--Tachikawa correspondence. Given a Morita--Tachikawa triple $(A,E,\Gamma)$, Iyama--Solberg show \cite[Thm.~4.5]{IS} that the algebra $\Gamma$ is minimal $d$-Auslander--Gorenstein if and only if $E$ is $d$-precluster-tilting \cite[Defn.~3.2]{IS}, and an earlier result of Iyama \cite[Thm.~4.5]{IS} shows that $\Gamma$ is a $d$-Auslander algebra if and only if $E$ is $d$-cluster-tilting, meaning that $\add{E}$ is maximal $(d-1)$-orthogonal \cite[Defn.~2.2]{IyHD}.

As promised, we may characterise $d$-Auslander and minimal $d$-Auslander--Gorenstein algebras via their shifted and coshifted modules, as follows.

\begin{thm}
\label{tilt-cotilt-dAG}
Let $\Gamma$ be a finite-dimensional non-selfinjective algebra, and write $d+1=\domdim{\Gamma}$. The following are equivalent:
\begin{itemize}\itemsep0em 
\item[(i)]$\Gamma$ is a minimal $d$-Auslander--Gorenstein algebra,
\item[(ii)]$\shiftmod{k}=\coshiftmod{d+1-k}$ for all $0\leq k\leq d+1$, and
\item[(iii)]there exists $M\in\lmod{\Gamma}$ that is both a shifted and a coshifted module.
\end{itemize}
Under these conditions, $\Gamma$ is a $d$-Auslander algebra if and only if $\gldim{\Gamma}<\infty$.
\end{thm}

\begin{proof}
We start by showing that (i) implies (ii), so assume that $\Gamma$ is minimal $d$-Auslander--Gorenstein. The assumptions on the homological dimensions of $\Gamma$ imply that the regular module has a minimal injective resolution
\[\begin{tikzcd}[column sep=20pt]
0\arrow{r}&\Gamma\arrow{r}&\Pi_0\arrow{r}&\cdots\arrow{r}&\Pi_d\arrow{r}&I\arrow{r}&0
\end{tikzcd}\]
with each $\Pi_j$ projective-injective. The number of indecomposable summands of $I$ is equal to the number of non-injective indecomposable summands of $\Gamma$ (cf.\ \cite[Thm.~5.2]{AR-CM}) and so, assuming without loss of generality that $\Gamma$ is basic, the indecomposable direct summands of $I$ are the indecomposable non-projective injective $\Gamma$-modules, each appearing with multiplicity one. It follows that we have $\kdual\Gamma=I\oplus\Pi$ for $\Pi$ the maximal projective-injective summand of $\Gamma$. Thus, by adding the identity map $\Pi\to\Pi$ to the right-hand end of the above injective resolution, we obtain a sequence
\[\begin{tikzcd}[column sep=20pt]
0\arrow{r}&\Gamma\arrow{r}&\Pi_0\arrow{r}&\cdots\arrow{r}&\Pi_d\arrow{r}&\kdual\Gamma\arrow{r}&0
\end{tikzcd}\]
in which each $\Pi_j$ is projective-injective. This is simultaneously an injective resolution of $\Gamma$ and a projective resolution of $\kdual\Gamma$, and has the appropriate number of projective-injective terms for computing shifted and coshifted modules, so these modules must coincide as claimed.

Since (ii) trivially implies (iii), it remains to show that (iii) implies (i). Let $M\cong\shiftmod{m}\cong\coshiftmod{n}$. Then $M$ has a projective resolution of the form
\[\begin{tikzcd}[column sep=20pt]
0\arrow{r}&\Gamma\arrow{r}&\Pi_{m-1}\arrow{r}&\cdots\arrow{r}&\Pi_0\arrow{r}&M\arrow{r}&0,
\end{tikzcd}\]
where $\Pi_j\in\add{\Pi}$ for each $j$, and an injective resolution of the form
\[\begin{tikzcd}[column sep=20pt]
0\arrow{r}&M\arrow{r}&\Pi_m\arrow{r}&\cdots\arrow{r}&\Pi_{m+n-1}\arrow{r}&\kdual\Gamma\arrow{r}&0,
\end{tikzcd}\]
with $\Pi_j\in\add{\Pi}$ for each $j$. Taking the Yoneda product of these two sequences produces a sequence
\[\begin{tikzcd}[column sep=20pt]
0\arrow{r}&\Gamma\arrow{r}&\Pi_0\arrow{r}&\cdots\arrow{r}&\Pi_{m+n-1}\arrow{r}&\kdual\Gamma\arrow{r}&0,
\end{tikzcd}\]
which shows that $\id\Gamma\leq m+n\leq\domdim{\Gamma}$, i.e.\ that $\Gamma$ is minimal $(m+n-1)$-Auslander--Gorenstein. We also see from this sequence that $m+n=\domdim{\Gamma}=d+1$, else $\kdual\Gamma$ would be projective, contradicting our assumption that $\Gamma$ is not self-injective.

The final statement is the previously noted fact that $d$-Auslander algebras are precisely minimal $d$-Auslander--Gorenstein algebras of finite global dimension \cite[Prop.~4.8]{IS}.
\end{proof}

We remark that this result remains morally true when $\Gamma$ is selfinjective; in this case properties (i)--(iii) hold for any positive integer $d$. Combining it with Proposition~\ref{gen-cogen-characterisation}, we obtain the following corollary, the statement of which is more directly comparable with \cite[Lem.~1.1]{CBSa} for ($1$-)Auslander algebras and \cite[Thm.~2.4.11]{NRTZ} for minimal $1$-Auslander--Gorenstein algebras.

\begin{cor}
\label{democratic-dAG}
An algebra $\Gamma$ is minimal $d$-Auslander--Gorenstein if and only if for some (or equivalently every) $m,n\geq0$ such that $d=m+n-1$, there is a $\Gamma$-module in $\gen_{m-1}(\Pi)\cap\cogen^{n-1}(\Pi)$ that is $m$-tilting and $n$-cotilting,  where $\Pi$ is the maximal projective-injective summand of $\Gamma$.
\end{cor}

\section{Recollements and homotopy categories}
\label{homotopy-cats}
Given a Morita--Tachikawa triple $(A,E,\Gamma)$, the module category of each shifted and coshifted algebra of $\Gamma$ is naturally part of a recollement, also involving $\lmod{A}$. Before describing and discussing these specific recollements, we will recall some facts about idempotent recollements in general.

\subsection{Idempotent recollements}
Let $B $ be a finite-dimensional algebra, let $e\in B$ be an idempotent element and write $A=eBe$ for the corresponding idempotent subalgebra (sometimes called the \emph{corner} or \emph{boundary algebra}). We obtain from $e$ a diagram
\begin{equation}
\label{recollement}
\begin{tikzcd}[column sep=40pt]
\lmod{B/BeB}\arrow["i" description]{r}&\lmod{B}\arrow["e" description]{r}\arrow[shift right=3,swap]{l}{q}\arrow[shift left=3]{l}{p}&\lmod{A}\arrow[shift right=3,swap]{l}{\ell}\arrow[shift left=3]{l}{r}
\end{tikzcd}
\end{equation}
of six functors, defined by 
\begin{align*}
q & = B/BeB\otimes_{B}-,  & \ell &= Be\otimes_A -,  \\
i & = \Hom_{B/BeB}(B/BeB,-) = B/BeB \otimes_{B/BeB}-,   &  e &= \Hom_{B}(Be,-)=eB\otimes_B-,   \\
p & = \Hom_{B}(B/BeB, -),  & r &= \Hom_{A} (eB,- ).
\end{align*}
Various properties of the above six functors, including that both $(\ell,e)$ and $(e,r)$ are adjoint pairs, mean that this diagram forms a \emph{recollement} of abelian categories; we do not give the general definition here since we will only consider recollements of module categories determined by idempotents as above (cf.\ \cite{PV}). For a $\Gamma$-module $M$, one obtains the same $A$-module $eM$ either by applying the functor $e$ in this diagram, or by multiplying on the left by the idempotent $e$, hence the abuse of notation.

Write $\eta^\ell\colon 1\to e\ell$ and $\varepsilon^\ell\colon\ell e \to 1$ for the unit and counit of the adjunction $(\ell,e)$, and similarly $\eta^r$ and $\varepsilon^r$ for the unit and counit of the adjunction $(e,r)$. A special property of idempotent recollements is that the unit $\eta^\ell$ and the counit $\varepsilon^r$ are natural isomorphisms. This means that there is a natural isomorphism
\[\Hom_B(\ell M,rM)\isoto\Hom_A(M,M),\]
functorial in $M$, determining a canonical map of functors $\zeta\colon\ell\to r$. 
Indeed, on an object $M$, the natural transformation $\zeta$ is given by either of the compositions in the commutative diagram
\[\begin{tikzcd}[column sep=20pt]
&re\ell M\arrow{dr}{r(\eta^\ell_M)}[swap]{\sim}&\\
\ell M\arrow{rr}{\zeta_M}\arrow{ur}{\eta_{\ell M}^r}\arrow{dr}{\sim}[swap]{\ell(\varepsilon^r_M)}&&rM\\
&\ell e rM\arrow{ur}[swap]{\varepsilon_{rM}^\ell}&
\end{tikzcd}\]
Moreover, $\zeta_M$ is characterised by the property that $\varepsilon^r_M\circ(e\zeta_M)\circ\eta^\ell_ M=1_M\colon M\to M$.

Taking the image of $\zeta$ yields a seventh functor $c\colon \lmod{A} \to \lmod{B}$, called the intermediate extension \cite{Ku}, which, like $\ell$ and $r$, is fully faithful. In the sequel, we will implicitly use the natural epimorphism $\ell\to c$ and monomorphism $c\to r$ composing to the natural map $\zeta$.

Since $\ell, r$ and $c$ are fully faithful and $er\cong 1\cong e\ell $, we also have $ec \cong 1$, and we obtain three induced equivalences of categories  
\[\begin{tikzcd}[row sep=10pt]
\Bild{\ell}\\
\Bild{c}&\lmod{A}\arrow{ul}[swap]{\ell}\arrow{l}{c}\arrow{dl}{r}\\
\Bild{r}
\end{tikzcd}\]
with quasi-inverses given by the respective restrictions of the functor $e$.
On the other side of the recollement, the functor $i$ embeds $\lmod{B/BeB}$ into $\lmod{B}$, and since $pi\cong 1\cong qi$ we see that the restrictions of $q$ and of $p$ to $\Bild i$ are both quasi-inverse to $i$.

The recollement \eqref{recollement} determines a \emph{TTF-triple} in $\lmod{B}$, meaning a triple $(\mcX,\mcY,\mcZ)$ of subcategories such that both $(\mcX,\mcY)$ and $(\mcY,\mcZ)$ are torsion pairs, by 
\[ 
{\TTF}(e)=(\mcX(e),\mcY(e),\mcZ(e)):= (\Ker q, \Ker e , \Ker p).
\]
We now give some alternative descriptions of the kernels and images of the functors in our recollement \eqref{recollement}, including the categories $\Ker{q}$ and $\Ker{p}$ appearing in this TTF-triple, in terms of the categories $\gen_k(X)$ and $\cogen^k(X)$ associated to $X\in\lmod{B}$ as in Definition~\ref{gen-cogen}.

\begin{lem} \label{ImageAndKernel}
For $B$ and $e$ as in \eqref{recollement}, write $P=Be$ and $I=\nu P=\kdual(eB)$. We have
\[
\begin{aligned}
 \Ker q &= \gen (P),  &&& \Bild \ell& =\gen_1 (P), \\
\Ker p &= \cogen (I), &&& \Bild r &=\cogen^1 (I).
\end{aligned}
\]
Moreover, the image of the intermediate extension $c=\Bild(\ell\to r)$ is given by
\[ 
\Bild c = \Ker p \cap \Ker q = \gen (P) \cap \cogen (I).
\]
\end{lem}

\begin{proof} For the equalities $\Bild \ell = \gen_1 (P)$ and $\Bild r = \cogen^1(I)$, see \cite[Lem.~3.1]{APT}. By \cite[Lem/Def.~2.4]{CBSa}, if $X\in \Ker q$ then the counit map $\ell e X\to X$ is an epimorphism. Take a projective cover $Q\to eX$; since $\ell $ preserves epimorphisms we obtain an epimorphism $\ell Q\to \ell e X \to X$. Since $\ell A = P$, we have $\ell Q \in \add P$ and thus $X\in \gen (P)$. Conversely, $\gen (P) \subseteq \Ker q$ since $q P= q\ell A =0$ and $q$ preserves epimorphisms. Using instead \cite[Lem/Def.~ 2.3]{CBSa}, one similarly proves that $\Ker p = \cogen (I)$. Finally, the equality $\Bild c = \Ker p \cap \Ker q $ is the first statement of \cite[Prop.~4.11]{FP}. 
\end{proof}

Now let $(A,E,\Gamma)$ be a Morita--Tachikawa triple, with $\Pi$ the maximal projective-injective summand of $\Gamma$. Recall from the Morita--Tachikawa correspondence that $A\cong\End_\Gamma(\Pi)^{\op}$. If $T$ is any tilting (or cotilting) $\Gamma$-module, we must have $\Pi\in\add{T}$. It follows that there is an idempotent $e\in B=\End_\Gamma(T)^{\op}$, given by projection onto the summand $\Pi$ of $T$, such that
\[eBe=\End_\Gamma(\Pi)^{\op}\cong A.\]
Thus we get a recollement as in \eqref{recollement}. In particular, this holds for the shifted and coshifted algebras $\shiftalg{k}$ and $\coshiftalg{k}$ of $\Gamma$. In this section, we explain how these different recollements are related, for different values of $k$, and give an explicit formula for the intermediate extension functor in each case.

\subsection{Recollements for shifted and coshifted algebras}
\label{shift-recollements}

We first introduce some notation for our preferred idempotents. Let $\Gamma$ be a finite-dimensional algebra and $k\leq\domdim\Gamma$. We denote by $e_k$ the idempotent of the $k$-th shifted algebra $\shiftalg{k}$ of $\Gamma$ given by projection onto $\Pi\in\add{\shiftmod{k}}$, and by $e^k$ the idempotent of the $k$-th coshifted algebra $\coshiftalg{k}$ given by projection onto $\Pi\in\add{\coshiftmod{k}}$.

\begin{rem}
\label{two-idempotents}
The reader is warned that while we have natural isomorphisms $B_0\isoto\Gamma\isoto B^0$, the composition does not take $e_0$ to $e^0$. Rather, identifying $e_0$ and $e^0$ with their respective image and preimage in $\Gamma$, we have isomorphisms $\Gamma e_0\cong\Pi\cong\kdual(e^0\Gamma)$, so that $e_0$ can be read off from the top of $\Pi$, whereas $e^0$ is read off from the socle.
\end{rem}



The algebras $e_k\shiftalg{k}e_k$ and $e^k\coshiftalg{k}e^k$ are all isomorphic to $A:=\End_\Gamma(\Pi)^{\op}$, so $\lmod{A}$ appears on the right-hand side of all of our recollements. In the case of the quotient algebras $\shiftalg{k}/\shiftalg{k}e_k\shiftalg{k}$ and $\coshiftalg{k}/\coshiftalg{k}e^k\coshiftalg{k}$ appearing on the other side of the recollements, we have the following.

\begin{lem}
\label{quotient-algebras}
For all $0\leq k\leq \domdim{\Gamma}$ we have isomorphisms
\begin{align*}
\shiftalg{k}/\shiftalg{k}e_k\shiftalg{k} &\cong \Gamma /\Gamma e_0\Gamma,\\
\coshiftalg{k}/\coshiftalg{k}e^k\coshiftalg{k}  &\cong \Gamma /\Gamma e^0\Gamma,
\end{align*}
induced by taking syzygies and cosyzygies.
\end{lem}

\begin{proof}
The idempotents $e_k$ are chosen such that there is an isomorphism
\[\shiftalg{k}/\shiftalg{k}e_k\shiftalg{k} \cong \End_{\lmod{\Gamma}/\add{\Pi}}(\shiftmod{k}).\]
Moreover, since $\Pi$ is projective-injective, \cite[Thm.~5.2]{AR-CM} provides mutually inverse equivalences
\[\Omega\colon\gen(\Pi)/\add{\Pi}\stackrel{\sim}{\longleftrightarrow}\cogen(\Pi)/\add{\Pi}\colon\Omega^-,\]
where $\Omega(X)$ is the kernel of a minimal projective cover of $X$, and $\Omega^-(Y)$ is the cokernel of a minimal injective hull of $Y$; for $X\in\gen(\Pi)$, such a projective cover is a minimal left $\add{\Pi}$-approximation as referred to in \cite[Thm.~5.2]{AR-CM}, and the corresponding statement holds for $Y\in\cogen(\Pi)$.

Now, noting that for $k=0$ there is nothing to prove, the result for $1\leq k\leq\domdim{\Gamma}$ follows inductively using the fact that, by construction, $\shiftmod{k}\in\gen(\Pi)$ and $\Omega(\shiftmod{k})$ agrees with $\shiftmod{k-1}$ up to a summand in $\add{\Pi}$, i.e.\ $\Omega(\shiftmod{k})\cong \shiftmod{k-1}\in\lmod{\Gamma}/\add{\Pi}$. The result for $\coshiftmod{k}$ is proved dually.

\end{proof}

It follows from Lemma~\ref{quotient-algebras} that the families of shifted and coshifted modules each provide a family of recollements, such that the left-hand side of the recollement is constant in each family, and the right-hand side is constant across both families. More precisely, for each $0\leq k\leq \domdim{\Gamma}$, we get a pair of recollements as follows.
\begin{equation}
\label{shifted-recollements}
\begin{tikzcd}[column sep=50pt,row sep=15pt]
\lmod{\Gamma/\Gamma e_0\Gamma}\arrow["i_k" description]{r}&\lmod{\shiftalg{k}}\arrow["e_k" description, shift left=3]{dr}\arrow[shift right=3,swap]{l}{q_k}\arrow[shift left=3]{l}{p_k}\\
&&\lmod{A}\arrow[shift right=6,swap]{ul}{\ell_k}\arrow{ul}{r_k}\arrow[swap]{dl}{\ell_k}\arrow[shift left=6]{dl}{r^k}\\
\lmod{\Gamma/\Gamma e^0\Gamma}\arrow["i^k" description]{r}&\lmod{\coshiftalg{k}}\arrow["e^k" description,shift right=3]{ur}\arrow[shift right=3,swap]{l}{q^k}\arrow[shift left=3]{l}{p^k}
\end{tikzcd}
\end{equation}
We denote the intermediate extension functors in these recollements by $c_k$ and $c^k$ respectively.
\subsection{Homotopy categories}
We now turn to the problem of computing the intermediate extension functor in each recollement from \eqref{shifted-recollements}. To do this, it will be useful to give a new description of the shifted and coshifted algebras as endomorphism algebras in the bounded homotopy category of $A$-modules, rather than in the category of $\Gamma$-modules, generalising a result of Crawley-Boevey and the second author \cite[Prop.~5.5]{CBSa} in the case that $\Gamma$ is an Auslander algebra. Our proof is also somewhat simpler and more conceptual, using only standard homological algebra.

We begin with the following general considerations. Let $A$ be a finite-dimensional algebra, $E\in\lmod{A}$, and $\Gamma=\End_A(E)^{\op}$. The bounded homotopy categories $\homot[\bound]{\proje{\Gamma}}$ and $\homot[\bound]{\injec{\Gamma}}$ of complexes of projective and injective $\Gamma$ modules respectively admit tautological functors to the unbounded derived category $\dercat{\Gamma}$, equivalences onto their images, which we treat as identifications. These subcategories may be characterised intrinsically as the full subcategories of $\dercat{\Gamma}$ on the compact and cocompact objects (in the context of additive categories) respectively. Extending the Yoneda equivalences
\begin{align*}
\Hom_A(E,-)&\colon\add{E}\isoto\proje{\Gamma},\\
\kdual\Hom_A(-,E)&\colon\add{E}\isoto\injec{\Gamma}
\end{align*}
to complexes, one sees that both of these subcategories of $\dercat{\Gamma}$ are equivalent to the full subcategory $\thick(E)$ of $\homot[\bound]{A}$, i.e.\ the smallest triangulated subcategory of the homotopy category $\homot[\bound]{A}$ closed under direct summands and containing (the stalk complex) $E$.

Now let $F\colon\mcT\isoto\dercat{\Gamma}$ be any equivalence of triangulated categories. Using the intrinsic description of $\homot[\bound]{\proje{\Gamma}}$ and $\homot[\bound]{\injec{\Gamma}}$ above, we see that $F$ induces respective equivalences from the subcategories of compact and cocompact objects of $\mcT$ to $\homot[\bound]{\proje{\Gamma}}$ and $\homot[\bound]{\injec{\Gamma}}$ respectively, and thus realises $\thick{E}$ as a full subcategory of $\mcT$ (in two ways). This holds in particular when $\mcT=\dercat{B}$ for some algebra $B$ derived equivalent to $\Gamma$, such as the endomorphism algebra of a tilting or cotilting $\Gamma$-module.

Given a derived equivalence $\dercat{B}\isoto\dercat{\Gamma}$, it follows from Rickard's Morita theory for derived categories \cite{RickMT} that the image of the stalk complex $B$ in $\homot[\bound]{\proje{\Gamma}}$ is a tilting complex for $\Gamma$ with endomorphism algebra $B$. The preimage of this complex under the Yoneda equivalence is an object of $\thick{E}\subseteq\homot[\bound]{A}$, again with endomorphism algebra $B$. Similarly, the image of $\kdual B\in\homot[\bound]{\injec{B}}$ in $\homot[\bound]{\injec{\Gamma}}$ is a cotilting complex, related by the dual Yoneda equivalence to another object of $\thick{E}$ with endomorphism algebra $B$.

Our conclusion is that when $\Gamma$ is the endomorphism algebra of an $A$-module $E$ (or more generally an object $E\in\homot[\bound]{A}$), any algebra $B$ derived equivalent to $\Gamma$ must also appear as an endomorphism algebra in $\thick{E}\subseteq\homot[\bound]{A}$. In general, $B$ need not be an endomorphism algebra in $\lmod{A}$. When $E$ is a generator-cogenerator and $B$ is one of the shifted or coshifted algebras of $\Gamma$, we may compute the relevant objects of $\thick{E}$ explicitly, and obtain a particularly straightforward answer.

\begin{thm}
\label{HtoB}
Let $(A,E,\Gamma)$ be a Morita--Tachikawa triple with all objects basic, and let $0\leq k\leq\domdim{\Gamma}$. 
\begin{itemize}
\item[(a)]Write $E^k=(P_{k-1}\to\cdots\to P_0\to E)\oplus A[k]\in\homot[\bound]{A}$, 
where the first summand denotes the complex whose non-zero part is given by the first $k$ terms of a minimal projective resolution of $E$, with $E$ in degree $0$, and the second denotes the stalk complex with $A$ in degree $-k$. Then
\[\coshiftalg{k}\cong\End_{\homot[\bound]{A}}(E^k)^{\op},\]
with the idempotent $e^k\in\coshiftalg{k}$ corresponding to projection onto the summand $A[k]$.
\item[(b)]Write $E_k=(E\to Q_0\to\cdots\to Q_{k-1})\oplus\kdual A[-k]$, where the first summand denotes the complex whose non-zero part is given by the first $k$ terms of a minimal injective resolution of $E$, with $E$ in degree $0$, and the second denotes the stalk complex with $\kdual A$ in degree $k$. Then
\[\shiftalg{k}\cong\End_{\homot[\bound]{A}}(E_k)^{\op},\]
with the idempotent $e_k\in\shiftalg{k}$ corresponding to projection onto the summand $\kdual A[-k]$.
\end{itemize}
\end{thm}
\begin{proof}
As usual, we only prove (a), since (b) is dual. By definition, $\coshiftalg{k}$ is the endomorphism algebra of the $k$-cotilting $\Gamma$-module $\coshiftmod{k}$, so that the image of $\kdual\coshiftalg{k}$ in $\homot[\bound]{\injec{\Gamma}}$ is given by an injective resolution of $\coshiftmod{k}$. By the preceding discussion, we need only show that the dual Yoneda equivalence maps $E^k$ to such a resolution (up to a degree shift). But this equivalence sends $E^k$ to the complex
\begin{equation}
\label{coshift-injres}
\begin{tikzcd}[column sep=20pt]
\Pi^{k-1}\oplus\Pi\arrow{r}&\Pi^{k-2}\arrow{r}&\cdots\arrow{r}&\Pi^0\arrow{r}&\kdual\Gamma
\end{tikzcd}
\end{equation}
where
\[\begin{tikzcd}[column sep=20pt]
\Pi^{k-1}\arrow{r}&\Pi^{k-2}\arrow{r}&\cdots\arrow{r}&\Pi^0\arrow{r}&\kdual\Gamma\arrow{r}&0,
\end{tikzcd}\]
begins a minimal projective resolution of $\kdual\Gamma$, and $\Pi$ denotes as usual the maximal projective-injective summand of $\Gamma$. This complex is exact except in degree $-k$, and by comparing to the definition of the coshifted modules, we see that its cohomology in this degree is precisely $C^k$, as required.
\end{proof}
\begin{rem}
When $\Gamma$ is an Auslander algebra, so $A$ is representation-finite and $\add{E}=\lmod{A}$, the category $\add{E^1}$ is equivalent to the category $\mcH$ from \cite[\S3]{CBSa}. Moreover, $\coshiftmod{1}=\shiftmod{1}$, and so Theorem~\ref{HtoB}(a) recovers \cite[Prop.~5.5]{CBSa} in this case. In contrast to the proof given in \cite{CBSa}, we did not have to identify the module over $\End_{\homot[\bound]{A}}(E^1)^{\op}$ corresponding to the $\coshiftalg{1}$-module $\kdual\coshiftmod{1}$.
\end{rem}

\subsection{Intermediate extensions}
As in \cite{CBSa}, the advantage of describing $\coshiftalg{k}$ as the endomorphism algebra of a complex, as in Theorem~\ref{HtoB}, is that it allows for convenient descriptions of some of the functors in the recollements \eqref{shifted-recollements}, as we will demonstrate in this section. We will also state the dual results for $\shiftalg{k}$.
Throughout, we treat the isomorphisms of Theorem~\ref{HtoB}, as well as the natural isomorphisms
\[\End_{\homot[\bound]{A}}(A[k])^{\op}\cong A\cong \End_{\homot[\bound]{A}}(\kdual A[-k])^{\op}\]
with which they are compatible, as identifications. We also write $f^k\colon A\oplus P_{k-1}\to P_{k-2}$ for the leftmost non-zero map in the complex $E^k$ (see Theorem~\ref{HtoB}(a)); this notation includes, in the case $k=1$, the convention that $P_{-1}=E$ (and so this object is typically not projective).

\begin{lem}
\label{lrdescription}For $M\in\lmod{A}$, we have
\[
\begin{aligned}
\ell^k (M) &\cong \Hom_{\homot[\bound]{A}}(E^k,A[k])\otimes_AM,\\
r^k(M)&\cong \Hom_A(\Ker f^k, M)\cong\Hom_{\dercat[\bound]{A}}(E^k,M[k]).
\end{aligned}
\]
\end{lem}
\begin{proof}
Under the isomorphisms of Theorem~\ref{HtoB}, we have
\[e^k=\Hom_{\coshiftalg{k}}(\Hom_{\homot[\bound]{A}}(E^k,A[k]),-)=\Hom_{\homot[\bound]{A}}(A[k],E^k)\otimes_{\coshiftalg{k}}-,\]
recalling that these isomorphisms identify the idempotent $e^k$ with projection onto the summand $A[k]$ of $E^k$. Thus the adjoints to $e^k$ arise from usual tensor-hom adjunction, which immediately gives the required formula
\[\ell^k (M) \cong \Hom_{\homot[\bound]{A}}(E^k,A[k])\otimes_AM,\]
for the left adjoint, and the formula
\[r^k(M)\cong\Hom_A(\Hom_{\homot[\bound]{A}}(A[k],E^k),M)\]
for the right. By a  standard computation in the homotopy category, we have an $A$-module isomorphism
\[\Hom_{\homot[\bound]{A}}(A[k],E^k)\cong\Ker{f^k},\]
providing the first claimed formula for $r^k$. The second then follows by observing, directly from the definition, that $E^k$ is isomorphic to the stalk complex $\Ker{f^k}[k]$ in $\dercat[\bound]{A}$.
\end{proof}

\begin{lem}
\label{ldescription2}
For $M\in\lmod{A}$ and $k\geq 2$, we have
\[\ell^k (M) = \coKer (\Hom_A(P_{k-2}, M) \to \Hom_A(A\oplus P_{k-1}, M) )
 =\Hom_{\homot[\bound]{A}}(E^k,M[k]).\]
\end{lem}
\begin{proof}
Let $P^M_1\to P^M_0\to M\to0$ be a projective presentation of $M$, whence we obtain the exact sequence
\[\begin{tikzcd}[column sep=10pt]
\Hom_{\homot[\bound]{A}}(E^k,A[k])\otimes_AP^M_1\arrow{r}&\Hom_{\homot[\bound]{A}}(E^k,A[k])\otimes_AP^M_0\arrow{r}&\Hom_{\homot[\bound]{A}}(E^k,A[k])\otimes_AM\arrow{r}&0.
\end{tikzcd}\]
We have $\Hom_{\homot[\bound]{A}}(E^k,A[k])\otimes_AM=\ell^k(M)$ by Lemma~\ref{lrdescription}, and there are natural isomorphisms $\Hom_{\homot[\bound]{A}}(E^k,A[k])\otimes_AP^M_i\isoto\Hom_{\homot[\bound]{A}}(E^k,P^M_i[k])$, 
identifying $\ell^k(M)$ with the cokernel of the map $\Hom_{\homot[\bound]{A}}(E^k,P^M_1[k])\to\Hom_{\homot[\bound]{A}}(E^k,P^M_0[k])$.

For any $N\in\lmod{A}$, we may compute $\Hom_{\homot[\bound]{A}}(E^k,N[k])$ via the exact sequence
\[\begin{tikzcd}[column sep=20pt]
\Hom_A(P_{k-2},N)\arrow{r}&\Hom_A(A\oplus P_{k-1},N)\arrow{r}&\Hom_{\homot[\bound]{A}}(E^k,N[k])\arrow{r}&0.
\end{tikzcd}\]
From this observation and our projective presentation of $M$, we may construct the commutative diagram
\[\begin{tikzcd}
\Hom_A(P_{k-2},P^M_1)\arrow{r}\arrow{d}&\Hom_A(P_{k-2},P^M_0)\arrow{r}\arrow{d}&\Hom_A(P_{k-2},M)\arrow{r}\arrow{d}&0\\
\Hom_A(A\oplus P_{k-1},P^M_1)\arrow{r}\arrow{d}&\Hom_A(A\oplus P_{k-1},P^M_0)\arrow{r}\arrow{d}&\Hom_A(A\oplus P_{k-1},M)\arrow{r}\arrow{d}&0\\
\Hom_{\homot[\bound]{A}}(E^k,P^M_1[k])\arrow{r}\arrow{d}&\Hom_{\homot[\bound]{A}}(E^k,P^M_0[k])\arrow{r}\arrow{d}&\Hom_{\homot[\bound]{A}}(E^k,M[k])\arrow{r}\arrow{d}&0\\
0&0&0
\end{tikzcd}\]
with exact columns. Since both $A\oplus P_{k-1}$ and $P_{k-2}$ are projective---the latter because of our assumption that $k\geq2$---the first two rows are exact. Exactness of the third row then follows from a variant of the snake lemma, and so $\ell^k(M)=\Hom_{\homot[\bound]{A}}(E^k,M[k])$ as claimed.
\end{proof}

\begin{thm}
\label{cdescription}
For $M\in\lmod{A}$, the intermediate extension $c^k(M)\in\lmod{\coshiftalg{k}}$ is given by 
\[c^k(M) = \coKer (\Hom_A(\Bild f^k, M) \to \Hom_A(A\oplus P_{k-1}, M)).\]
\end{thm}
\begin{proof}
We first deal separately with the case $k=1$. By \cite[Lem.~4.2]{CBSa}, \[c^1(M)=\coKer(\Hom_A(E,M)\to\Hom_A(P_0\oplus A,M)),\]
and $\Bild{f^1}=E$ in this case since $P_0\to E$ is a projective cover, giving the desired result.

Assume now that $k\geq2$, and denote by $(f^k)^*\colon \Hom_A(P_{k-2},M)\to\Hom_A(A\oplus P_{k-1},M)$ the map induced by $f^k$, so that $\ell^k(M)=\coKer{(f^k)^*}$ by Lemma~\ref{ldescription2}. Since $(f^k)^*$ factors through the inclusion $\Hom_A(\Bild{f^k},M)\to\Hom_A(A\oplus P_{k-1},M)$, we obtain a map $\zeta'\colon\ell^k(M)\to r^k(M)$ via the composition
\[
\ell^k (M) = \coKer{(f^k)^*} \surj \Hom_A(A\oplus P_{k-1}, M) /\Hom_A(\Bild f^k,M) \inj \Hom_A(\Ker f^k, M)= r^k(M),
\]
which we claim is the natural transformation $\zeta^k\colon\ell^k\to r^k$ evaluated on $M$. Recalling that $e^k\in\End_{\homot[\bound]{A}}(E^k)^{\op}$ is projection onto $A[k]$, one can check that $e^k(\zeta')$ is the identity on $\Hom_A(A,M)$. Under our isomorphism $\coshiftalg{k}\cong\End_{\homot[\bound]{A}}(E^k)$ from Theorem~\ref{HtoB}, and the identification of $\ell^k$ and $r^k$ with their descriptions in Lemmas~\ref{lrdescription} and \ref{ldescription2}, the unit $\varepsilon$ of the adjunction $(\ell^k,e^k)$ is identified with the natural isomorphism $M\isoto\Hom_A(A,M)$, and the counit $\eta$ of the adjunction $(e^k,r^k)$ with its inverse. Thus $\eta\circ(e^k\zeta')\circ\varepsilon=1_M$, and so $\zeta'=\zeta^k_M$ as claimed.
It then follows that $c^k(M):=\Bild{\zeta^k_M}=\Bild{\zeta'}$ is
\[\Bild (\Hom_A(A\oplus P_{k-1}, M) \to \Hom_A(\Ker f^k, M))=\coKer (\Hom_A(\Bild f^k, M) \to \Hom_A(A\oplus P_{k-1}, M)),\]
as required.
\end{proof}

\begin{rem}
Using the descriptions
\begin{align*}
\ell^k(M)(X)=\Hom_{\homot[\bound]{A}}(X,M[k]),\\
r^k(M)(X)=\Hom_{\dercat[\bound]{A}}(X,M[k])
\end{align*}
of $\ell^k$ and $r^k$ when $k\geq2$, one can see that the canonical map $\zeta^k\colon\ell^k\to r^k$ agrees with that coming from the Verdier localisation functor $\homot[\bound]{A}\to\dercat[\bound]{A}$. Indeed, the isomorphism of $\Hom_{\homot[\bound]{A}}(X,M[k])$ with $\Hom_A(X_k,M)/\Bild{(f^k)^*}$ identifies the set of maps factoring through an acyclic complex, which is the kernel of the Verdier localisation functor, with $\Hom_A(\Bild{f^k},M)/\Bild{(f^k)^*}$.
\end{rem}

We now state the corresponding dual results for $\ell_k$, $r_k$ and $c_k$, using the notation $g_k\colon Q_{k-2}\to Q_{k-1}\oplus\kdual A$ for the rightmost non-zero map in $E_k$.

\begin{lem}
\label{lrdescription*}
For $M\in\lmod{A}$, we have
\begin{align*}
\ell_k(M)&=\kdual\Hom_A(M,\coKer{g_k})=\kdual\Hom_{\dercat[\bound]{A}}(M[-k],E_k),\\
r_k(M)&=\Hom_A(\Hom_{\homot[\bound]{A}}(\kdual A[-k],E_k),M).
\end{align*}
\end{lem}

\begin{lem}
\label{ldescription2*}
For $M\in\lmod{A}$ and $k\geq 2$, we have
\[r_k (M) = \Ker (\kdual\Hom_A(M,Q_{k-1}\oplus\kdual A) \xrightarrow{\kdual (g_k)_*} \kdual\Hom_A(M,Q_{k-2}) )
=\kdual\Hom_{\homot[\bound]{A}}(M[-k],E_k).\]
\end{lem}

\begin{thm}
\label{cdescription*}
For $M\in\lmod{A}$, the intermediate extension $c_k(M)\in\lmod{\shiftmod{k}}$ is given by
\[
c_k(M) = \Ker (\kdual\Hom_A(M,Q_{k-1}\oplus\kdual A) \to \kdual\Hom_A(M,\Bild{g_k})).
\]
\end{thm}

\section{Tilting subcategories for shifted modules}
\label{subcategories}

When two algebras are related via tilting, a result of Miyashita provides equivalences between various subcategories of their module categories. In this section, we will first recall this result, and then provide convenient descriptions of the relevant subcategories in the case of shifted and coshifted modules. In fact, our results will hold for arbitrary special tilting or cotilting modules, in the sense of Definition~\ref{dfn-tilting}.


To begin with, let $\Gamma$ be any finite-dimensional $k$-algebra, and let $T\in\lmod{\Gamma}$ be a tilting module of any finite projective dimension. We set $B:= \End_{\Gamma } (T)^{\op}$ and note that
$\kdual T$ is a $k$-cotilting left $B$-module.
We define, for $i\geq0$, subcategories 
\[
\mcT_i (T) := \bigcap_{j\neq i} \Ker \Ext^j_{\Gamma} (T, -)  \quad \text{ and }\quad 
\mcC_i (\kdual T ) := \bigcap_{j\neq i} \Ker \Ext^j_{B} (-, \kdual T), \]
which we refer to collectively as the \emph{tilting subcategories} associated to $T$. If $\pdim{T}=k$, both $\mcT_i(T)$ and $\mcC_i(\kdual T)$ are zero for $i>k$. These are the subcategories involved in Miyashita's equivalences, which are as follows.

\begin{thm}[{\cite[Thm.~1.16]{M}}]\label{tiltingEq}
For $0\leq i\leq k$, the functor $\Ext^i_\Gamma (T, -) \colon \mcT_i (T) \to \mcC_i(\kdual T)$ is an equivalence of categories with quasi-inverse $\kdual\Ext^i_{B} (-, \kdual T ) \colon \mcC_i(\kdual T) \to \mcT_i(T)$.  
\end{thm}

In the case $k=1$, in which $T$ is a classical tilting module, the pair $(\mcT_0(T),\mcT_1(T))$ is a torsion pair in $\lmod{\Gamma}$, and $(\mcC_1(\kdual T),\mcC_0(\kdual T))$ is a torsion pair in $\lmod{B}$. In this case Miyashita's result recovers Brenner--Butler's famous theorem \cite{BrenBut} (see also \cite[\S VI.3]{ASS}), stating that the torsion class in each of these pairs is equivalent to the torsion-free class in the other.

We will now, over the course of a lemma and three propositions, calculate the tilting subcategories for a special tilting or cotilting module. In each case the proof is provided for tilting modules, and can be dualised to provide an argument for cotilting modules.

\begin{lem}
\label{special-kerHom}
If $k\geq1$ and $T$ is a $P$-special $k$-tilting $\Gamma$-module for some projective $P$, then
\[\Ker\Hom_\Gamma(T,-)=\Ker\Hom_\Gamma(P,-).\]
Dually, if $C$ is an $I$-special $k$-cotilting $\Gamma$-module for some projective $I$ and $k\geq1$, then $\Ker\Hom_\Gamma(-,C)=\Ker\Hom_\Gamma(-,I)$.
\end{lem}
\begin{proof}
Since $k\geq1$, it follows directly from the definition that $P$ is a summand of $T$, so $\Ker\Hom_\Gamma(T,-)\subseteq\Ker\Hom_\Gamma(P,-)$, and that $T\in\gen(P)$, so $\Ker\Hom_\Gamma(P,-)\subseteq\Ker\Hom_\Gamma(T,-)$.
\end{proof}

\begin{pro}
\label{0-subcats}
If $k\geq 1$ and $T$ is a $P$-special $k$-tilting $\Gamma$-module for some projective $P$, then
\[\mcT_0(T)=\gen_{k-1}(P).\]
Dually, if $C$ is an $I$-special $k$-cotilting module for some injective $I$, then $\mcC_0(C)=\cogen^{k-1}(I)$.
\end{pro}
\begin{proof}
Assume $X\in \gen_{k-1}(P)$, so we have an exact sequence
\[\begin{tikzcd}[column sep=20pt]
0\arrow{r}& Y\arrow{r}& P^{k-1}\arrow{r}& \cdots \arrow{r}& P^0 \arrow{r}& X\arrow{r}& 0
\end{tikzcd}\]
with $P^i\in \add P$. Since $P\in\add{T}$ and $T$ is $k$-tilting, a standard homological argument with long exact sequences shows that for $j\geq1$ we have  
\[ 
\Ext^j_{\Gamma} (T, X) = \Ext^{k+j}_\Gamma(T, Y) = 0,
\]
so $X\in\mcT_0(T)$.

We prove the converse by induction on $k$. In the case $k=1$, note that $P$ is a direct summand of $T\in\gen{P}$, and hence $\gen(P)=\gen(T)$, and the latter coincides with $\mcT_0(T)$ (e.g.\ by \cite[Thm.~VI.2.5]{ASS}). Now let $T$ be $P$-special $k$-tilting for $k>1$, so that there is an exact sequence
\begin{equation}
\label{spec-tilt-induc}
\begin{tikzcd}[column sep=20pt]
0\arrow{r}&\Gamma\arrow{r}{\varphi}&P_{k-1}\arrow{r}&\cdots\arrow{r}&P_0\arrow{r}{\psi}&T^\circ\arrow{r}&0
\end{tikzcd}
\end{equation}
with $P_i\in\add{P}$ and $T^\circ\in\add{T}$. It follows directly from this sequence that $T'=P\oplus\Ker{\psi}$ is $P$-special $(k-1)$-tilting, and that $T''=P\oplus\coKer{\varphi}$ is $P$-special $1$-tilting. By induction, we may assume that $\gen_{k-2}(P)=\mcT_0(T')$. Now let $X\in \mcT_0(T)$. It follows from \eqref{spec-tilt-induc} that $\mcT_0 (\shiftmod{k}) \subseteq \mcT_0(\shiftmod{k-1}) =\gen_{k-2}(\Pi )$, and so we have an exact sequence
\begin{equation}
\label{gen-comp}
\begin{tikzcd}[column sep=20pt]
0\arrow{r}&Z\arrow{r}&P^{k-2}\arrow{r}&\cdots\arrow{r}&P^0\arrow{r}&X\arrow{r}&0,
\end{tikzcd}
\end{equation}
with $P^i\in\add{P}$. Thus we only need to see that $Z\in \gen (P)$, or equivalently, by the base case of the induction, that $Z\in\mcT_0(T'')$. 
We claim that
\[ 
\Ext^1_\Gamma(T'',Z) \cong  \Ext^k_\Gamma (T, Z) \cong \Ext^1_\Gamma(T, X) =0.
\]
The first isomorphism follows from \eqref{spec-tilt-induc}, the second follows from \eqref{gen-comp}, and $\Ext^1_\Gamma(T,X)=0$ by assumption since $X\in\mcT_0(T)$. Thus $Z\in\ker\Ext^1_\Gamma(T'',-)=\mcT_0(T'')=\gen(P)$, as required.
\end{proof}

\begin{pro}
\label{int-subcats}
If $k\geq1$ and $T$ is a $P$-special $k$-tilting $\Gamma$-module for some projective $P$, then
\[\mcT_j(T)=\{0\}\]
for any $0<j<k$. Dually, if $C$ is an $I$-special $k$-cotilting module for some injective $I$, then $\mcC_j(C)=\{0\}$ for any $0<j<k$.
\end{pro}

\begin{proof}
Let $T''$ be the $P$-special $1$-tilting module defined as in the proof of Proposition~\ref{0-subcats}. As in this previous proof, it follows from the exact sequence \eqref{spec-tilt-induc} that
\[\Ker\Ext^k_\Gamma(T,-)=\Ker\Ext^1_\Gamma(T'',-)=\mcT_0(T''),\]
and so for $j\ne k$ we have $\mcT_j(T)\subseteq\mcT_0(T'')$.

On the other hand, if $j\ne0$ then we can use Lemma~\ref{special-kerHom} to see that
\[\mcT_j(T)\subseteq\Ker\Hom_\Gamma(T,-)=\Ker\Hom_\Gamma(P,-)=\Ker\Hom_\Gamma(T'',-)=\mcT_1(T'').\]
Thus $\mcT_j(T)\subseteq\mcT_0(T'')\cap\mcT_1(T'')$, but this intersection is $\{0\}$ since $T''$ is $1$-tilting, meaning that its two tilting subcategories form a torsion pair.
\end{proof}

\begin{pro}
\label{k-subcats}
If $k\geq1$ and $T$ is a $P$-special $k$-tilting $\Gamma$-module for some projective $P$, then
\[\mcT_k(T)=\Ker\Hom_\Gamma(P,-).\]
Dually, if $C$ is an $I$-special $k$-cotilting module for some injective $I$, then $\mcC_k(C)=\Ker\Hom_\Gamma(-,I)$.
\end{pro}
\begin{proof}
Since $k\geq1$, the inclusion
\[\mcT_k(T)\subseteq\Ker\Hom_\Gamma(T,-)=\Ker\Hom_\Gamma(P,-)\]
follows from Lemma~\ref{special-kerHom}, and so it remains to show the reverse inclusion.

Assume that $\Hom_\Gamma(P,X)=0$, and consider again the exact sequence \eqref{spec-tilt-induc}. It follows from this sequence that $P\oplus T^\circ$ is a $k$-tilting module, meaning that its additive closure, which is a priori contained in $\add(T)$, is even equal to $\add(T)$. Thus, to show that $\Ext^i_\Gamma(T,X)=0$ for $0<i<k$, it suffices to show that $\Ext^i_\Gamma(T^\circ,X)=0$ for such $i$. But writing $M_i$ for the kernel of the map in \eqref{spec-tilt-induc} starting at $P_i$ (and $M_{-1}=T^\circ)$, a standard homological argument shows that, for $0<i\leq k$, we have
\[\Ext^i_\Gamma(T^\circ,X)=\Ext^1_\Gamma(M_{i-2},X)=\coKer(\Hom_\Gamma(P_{i-2},X)\to\Hom_\Gamma(M_{i-1},X))=\Hom_\Gamma(M_{i-1},X)\]
since $\Hom_\Gamma(P,X)=0$ by assumption. Then, providing $i<k$, we see from \eqref{spec-tilt-induc} that $M_{i-1}\in\gen(P)$, hence $\Ker\Hom_\Gamma(M_{i-1},-)\subseteq\Ker\Hom_\Gamma(P,-)$ contains $X$. Thus $\Ext^i_\Gamma(T^\circ,X)=\Hom_\Gamma(M_{i-1},X)=0$, as required.
\end{proof}

We now apply the preceding results to the shifted and coshifted modules, to obtain the main results of this section.

\begin{thm}
\label{tilting-subcats}
Let $\Gamma$ be a finite-dimensional algebra with $\domdim{\Gamma}=d+1>0$ and maximal projective-injective module $\Pi$, and let $1\leq k\leq d+1$. Write $I_k= \kdual \Hom_{\Gamma} (\Pi,\shiftmod{k})$. Then $I_k$ is an injective summand of the $\shiftalg{k}$-module $\kdual\shiftmod{k}$, and the tilting subcategories associated to the shifted module $\shiftmod{k}$ are
\[\mcT_j(\shiftmod{k})=\begin{cases}\gen_{k-1}(\Pi),&j=0,\\\Ker\Hom_\Gamma(\Pi,-),&j=k,\\\{0\},&\text{otherwise},\end{cases}\qquad
\mcC_j(\kdual\shiftmod{k})=\begin{cases}\cogen^{k-1}(I_k),&j=0,\\\Ker\Hom_{\shiftalg{k}}(-,I_k),&j=k,\\\{0\},&\text{otherwise.}\end{cases}\]
\end{thm}

\begin{proof}
That $I_k$ is an injective summand of $\kdual\shiftmod{k}$ follows from the fact that $\Pi$ is a summand of both $\Gamma$ and $\shiftmod{k}$. The rest of the statement is a direct application of Propositions~\ref{0-subcats}, \ref{int-subcats} and \ref{k-subcats}, using that $\shiftmod{k}$ is $\Pi$-special $k$-tilting, and that $\kdual\shiftmod{k}$ is $I_k$-special $k$-cotilting (see Lemma~\ref{proj-injAndNakayama}(1)).\qedhere

\end{proof}

By combining Theorem~\ref{tilting-subcats} with Miyashita's equivalences from Theorem~\ref{tiltingEq}, we obtain the following.

\begin{cor}
In the setting of Theorem~\ref{tilting-subcats}, there are equivalences of categories
\begin{align*}
\Hom_\Gamma(\shiftmod{k},-)&\colon\gen_{k-1}(\Pi)\isoto\cogen^{k-1}(I_k)=\mcC_0(\kdual\shiftmod{k}),\\
\Ext^k_\Gamma(\shiftmod{k},-)&\colon\ker\Hom_\Gamma(\Pi,-)\isoto\Ker\Hom_{\shiftalg{k}}(-,I_k)=\mcC_k(\kdual\shiftmod{k}).
\end{align*}
In particular, the categories $\mcC_k(\kdual\shiftmod{k})$ for $1\leq k\leq d+1$ are all equivalent to each other, and if $k\geq2$ there is a fully faithful functor $\mcC_0(\kdual\shiftmod{k}) \to \mcC_0(\kdual\shiftmod{k-1})$ sending $\kdual\shiftmod{k}$ to $\kdual\shiftmod{k-1}$.   
\end{cor}
\begin{proof}
The equivalences are immediate from Theorems~\ref{tiltingEq} and \ref{tilting-subcats}. Using the second equivalence, we see that $\mcC_k(\kdual\shiftmod{k})\simeq\ker\Hom_\Gamma(\Pi,-)$ for any $1\leq k\leq d+1$. For the final statement, the fully faithful functor is provided by the composition
\[\mcC_0(\kdual\shiftmod{k})\isoto\gen_{k-1}(\Pi)\hookrightarrow\gen_{k-2}(\Pi)\isoto\mcC_0(\kdual\shiftmod{k-1}),\]
where the equivalences are those in Theorem~\ref{tiltingEq}, and the middle map is the natural inclusion. The first equivalence takes $\kdual\shiftmod{k}$ to $\kdual\Gamma$, which is then taken to $\kdual\shiftmod{k-1}$ by the second equivalence.
\end{proof}

Using the versions of Propositions~\ref{0-subcats}--\ref{k-subcats} for cotilting modules, we obtain the following dual statements.

\begin{thm}
Let $\Gamma$ be a finite-dimensional algebra with $\domdim{\Gamma}=d+1>0$ and maximal projective-injective module $\Pi$, and let $1\leq k\leq d+1$. Write $P^k= \Hom_{\Gamma} (\coshiftmod{k},\Pi)$. Then $P^k$ is a projective summand of the $\coshiftalg{k}$-module $\kdual\coshiftmod{k}$, and the tilting subcategories associated to the coshifted module $\coshiftmod{k}$ are
\[\mcC_j(\coshiftmod{k})=\begin{cases}\cogen^{k-1}(\Pi),&j=0,\\\Ker\Hom_\Gamma(-,\Pi),&j=k,\\\{0\},&\text{otherwise.}\end{cases}
\qquad
\mcT_j(\kdual\coshiftmod{k})=\begin{cases}\gen_{k-1}(P^k),&j=0,\\\Ker\Hom_{\coshiftalg{k}}(P^k,-),&j=k,\\\{0\},&\text{otherwise},\end{cases}
\]
\end{thm}

\begin{cor}
In the setting of Theorem~\ref{tilting-subcats}, there are equivalences of categories
\begin{align*}
\kdual\Hom_\Gamma(-,\coshiftmod{k})&\colon\cogen^{k-1}(\Pi)\isoto\gen_{k-1}(P^k)=\mcT_0(\kdual\coshiftmod{k}),\\
\kdual\Ext^k_\Gamma(-,\coshiftmod{k})&\colon\ker\Hom_\Gamma(-,\Pi)\to\Ker\Hom_{\coshiftalg{k}}(P^k,-)=\mcT_k(\kdual\coshiftmod{k}).
\end{align*}
In particular, the categories $\mcT_k(\kdual\coshiftmod{k})$ for $1\leq k\leq d+1$ are all equivalent to each other, and if $k\geq2$ there is a fully faithful functor $\mcT_0(\kdual\coshiftmod{k}) \to \mcT_0(\kdual\coshiftmod{k-1})$ sending $\kdual\coshiftmod{k}$ to $\kdual\coshiftmod{k-1}$.   
\end{cor}

\section{Tilting modules as intermediate extensions}
\label{int-exts}

As usual, let $\Gamma$ be a finite-dimensional algebra with $\domdim{\Gamma}=d+1>0$. In this section, we assume $d\geq1$, so that $\Gamma$ forms part of a Morita--Tachikawa triple $(A,E,\Gamma)$, and consider the intermediate extension functors in our preferred recollements involving the shifted and coshifted algebras $\shiftalg{k}$ and $\coshiftalg{k}$ of $\Gamma$, which we denote by $c_k$ and $c^k$ respectively. Our main result is that, for any $1\leq k\leq d$, the distinguished cotilting module $\kdual\shiftmod{k}$ for the $k$-th shifted algebra $\shiftalg{k}$ of $\Gamma$ is the intermediate extension $c_kE$. Similarly, $c^kE=\kdual\coshiftmod{k}$ is the distinguished tilting module for the coshifted algebra $\coshiftalg{k}$. 
We first give some general results, for arbitrary tilting or cotilting modules.

\begin{pro}
\label{rest-canonical-to-E}
Let $\Gamma$ be a finite-dimensional algebra with tilting module $T$, cotilting module $C$ and maximal projective summand $\Pi$, and write $B=\End_\Gamma(T)^{\op}$ and $B'=\End_\Gamma(C)^{\op}$. Let $e$ and $e'$ be the idempotents of $B$ and $B'$ given in each case by projection onto $\Pi$. Then there are natural isomorphisms
\[e\kdual T=\kdual\Pi=e'\kdual C\]
of $\End_\Gamma(\Pi)^{\op}$-modules. In particular, if $\Gamma$ is part of a Morita--Tachikawa triple $(A,E,\Gamma)$, then there are natural isomorphisms
\[e\kdual T=E=e'\kdual C\]
of $A$-modules.
\end{pro}
\begin{proof}
Writing $\Phi=\Hom_\Gamma(T,-)$, we have $Be=\Phi(\Pi)$ and $\kdual T=\Phi(\kdual\Gamma)$. It follows that
\[
e(\kdual T)  = \Hom_{B}(\Phi (\Pi), \Phi (\kdual \Gamma )) =\Hom_{\Gamma } (\Pi, \kdual \Gamma )= \kdual \Pi,\]
since by \cite[Thm.~1.16]{M} (here Theorem~\ref{tiltingEq}) $\Phi$ is fully faithful on the subcategory $\mcT_0(T)$, which contains all injective $\Gamma$-modules. Writing $\Phi'=\kdual\Hom_\Gamma(-,C)$, we have $\kdual(e'B')=\Phi'(\Pi)$ and $\kdual C=\Phi'(\Gamma)$. It follows that
\[e'(\kdual C)=\kdual\Hom_{B'}(\Phi'(\Gamma),\Phi'(\Pi))=\kdual\Hom_{\Gamma}(\Gamma,\Pi)=\kdual\Pi,\]
since by \cite[Thm.~1.16]{M} again, $\Phi'$ is fully faithful on the subcategory $\mcC_0(C)$, which contains all projective $\Gamma$-modules. The final statement follows since the module $E$ in a Morita--Tachikawa triple is always given by $\kdual\Pi$, for $\Pi$ as in the statement.
\end{proof}


Maintaining the notation of Proposition~\ref{rest-canonical-to-E}, consider the $B$-modules
\[
P= \Hom_{\Gamma}(T, \Pi ),\qquad I= \kdual \Hom_{\Gamma} (\Pi, T),
\]
noting that $P$ is projective, $I$ is injective and $\nu P =I$. Furthermore, since $\Pi$ is a summand of both $\Gamma$ and $\kdual\Gamma$, we have $P\oplus I \in \add\kdual T$. In terms of the idempotent $e$, we have $P=Be$ and $I=\kdual(eB)$. Our aim is now to characterise when the cotilting $B$-module $\kdual T$ is in the image of the intermediate extension functor $c$ associated to this idempotent.

\begin{pro} \label{CasIntExt} In the context of the preceding paragraph, let $m,n\geq0$, and denote by $\Omega$ and $\Omega^{-}$ the syzygy and cosyzygy functors for $\Gamma$.
\begin{itemize}
\item[(i)] The following are equivalent:
\begin{itemize} 
\item[(a)]
$\Gamma \in \cogen^{m-1}(\Pi )$ and $\Ext^1_{\Gamma} (\Omega^{-i} \Gamma, T)=0$ for $1\leq i \leq m$, and
\item[(b)] $\kdual T\in \cogen^{m-1}(I)$.
\end{itemize}
\item[(ii)] The following are equivalent:
\begin{itemize}
\item[(a)] $\kdual \Gamma\in \gen_{n-1} (\Pi )$ and $\Ext^1_{\Gamma} (T, \Omega^i \kdual \Gamma )=0$ for $1\leq i \leq n$, and
\item[(b)] $\kdual T\in \gen_{n-1} (P)$.
\end{itemize}
\end{itemize}
Moreover the conditions in (i) and (ii) both hold for some $m,n\geq1$ if and only if $\kdual T$ is in the image of the intermediate extension functor $c$ associated to $e$, and in this case $\kdual T=c(\kdual\Pi)$.
\end{pro}

\begin{proof}
Since conditions (a) and (b) are vacuous for $m=0$ and $n=0$  respectively, we may assume $m,n\geq1$. We will also use the following straightforward observations, which hold in an arbitrary abelian category. Given an exact sequence
\[X_{\bullet}=(\mathord{\begin{tikzcd}[column sep=20pt]
\cdots\arrow{r}& X_{i-1}\arrow{r}& X_i\arrow{r}& X_{i+1}\arrow{r}&\cdots
\end{tikzcd}}),\]
let $Z_i=\Ker (X_i\to X_{i+1})$ for each $i\in\Z$.
Then for any object $Y$,
\begin{itemize}
\item[(1)] if $\Ext^1(Y, Z_{i-1})=0$, then $\Hom (Y, X_{\bullet})$ is exact at $\Hom(Y,X_i)$, and
\item[(2)] if $\Ext^1(Z_{i+2}, Y) =0$, then $\Hom (X_{\bullet}, Y)$ is exact at $\Hom(X_i,Y)$.
\end{itemize}
The proof now proceeds as follows.
\begin{itemize}
\item[(i)]Assume $\Gamma\in\cogen^{m-1}(\Pi)$ and $\Ext^1_{\Gamma} (\Omega^{-i} \Gamma, T)=0$ for $1\leq i \leq m$. Consider an exact sequence
\[\begin{tikzcd}[column sep=20pt]
0\arrow{r}&\Gamma\arrow{r}&\Pi_0\arrow{r}&\cdots\arrow{r}&\Pi_{m-1}\arrow{r}& X\arrow{r}&0
\end{tikzcd}\]
with $\Pi_i\in\add{\Pi}$. Thinking of this as an infinite complex with $\Pi_i$ in degree $i$ and defining $Z_i$ as above, we can apply the functor $\Psi=\kdual\Hom_{\Gamma}(-,T)$ and use observation (2) to see that the resulting sequence
\[\begin{tikzcd}
0\arrow{r}&\Psi\Gamma\arrow{r}&\Psi\Pi_0\arrow{r}&\cdots\arrow{r}&\Psi\Pi_{k-1}
\end{tikzcd}\]
is exact, since $\Ext^1_\Gamma(Z_i,T)=\Ext^1_\Gamma(\Omega^{-i}\Gamma,T)=0$ for $1\leq i\leq m$. Since $\Psi(\Gamma)=\kdual T$ and $\Psi(\Pi)=I$, it follows that $\kdual T\in\cogen^{m-1}(I)$.


Conversely, assume $\kdual T\in\cogen^{m-1}(I)$, and take an exact sequence 
\[\begin{tikzcd}[column sep=20pt]
0\arrow{r}& \kdual T\arrow{r}& I_0\arrow{r}& I_1\arrow{r}& \cdots \arrow{r}& I_{m-1}\arrow{r}& Y \arrow{r}& 0
\end{tikzcd}\]
with each $I_i\in \add I$, viewed as an infinite complex with $I_i$ in degree $i$, and define $Z_i$ as above. Then a standard homological argument using the above sequence shows that
\[\Ext^1_B(\kdual T, Z_i)=\Ext^{i+1}_B(\kdual T,\kdual T)=0\]
for $0\leq i\leq m-1$. So by observation (1) we can apply the right adjoint $\Psi'=\Hom_B(\kdual T,-)$ of $\Psi$, which satisfies $\Psi'(\kdual T)=\Gamma$ and $\Psi'(I)=\Pi$, to get an exact sequence
\[\begin{tikzcd}[column sep=20pt]
0\arrow{r}& \Gamma \arrow{r}& \Pi_0\arrow{r}& \cdots \arrow{r}& \Pi_{m-1}\arrow{r}&\Psi'Y\arrow{r}&0.
\end{tikzcd}\]
It follows that $\Gamma\in\cogen^{m-1}(\Pi)$. This is also a projective resolution of $\Psi'Y$, so we can use it to compute $\kdual\Ext^i_\Gamma(\Psi'Y,T)$ by applying the right exact functor $\Psi$. However, applying this functor recovers the part
\[\begin{tikzcd}[column sep=20pt]
0\arrow{r}& \kdual T\arrow{r}& I_0\arrow{r}& I_1\arrow{r}& \cdots \arrow{r}& I_{m-1}
\end{tikzcd}\]
of the original exact sequence, since the natural map $\Psi\Psi'(\kdual T)\to \kdual T$ is an isomorphism and $I\in \add \kdual T$, so the cohomology of this complex vanishes in degrees $i\leq m-2$. On the other hand, the cohomology in degree $-1\leq i\leq m-2$ computes
\[\kdual\Ext^{m-1-i}_\Gamma(\Psi'Y,T)=\kdual\Ext^{m-1-i}_\Gamma(\Omega^{-m}\Gamma,T)=\kdual\Ext^1_\Gamma(\Omega^{-(i+2)}\Gamma,T),\]
and so $\Ext^1_\Gamma(\Omega^{-i} \Gamma , T)=0$ for $1\leq i\leq m$ as required.

\item[(ii)] This is analogous to (i), replacing $\Psi$ and $\Psi'$ by $\Phi =\Hom_{\Gamma}(T,-)$, with $\Phi (\kdual \Gamma ) =\kdual T$ and $\Phi (\Pi )=P$, and its left adjoint $\Phi^\prime = \kdual \Hom_B(-,\kdual T)$, with $\Phi^\prime (\kdual T)= \kdual \Gamma $ and $\Phi^\prime (P)=\Pi$.
\end{itemize}
If $\kdual T=c\kdual\Pi$, then $\kdual T\in\Bild{c}=\gen(P)\cap\cogen(I)$ by Lemma~\ref{ImageAndKernel}, so the conditions in (i) and (ii) hold for $m=n=1$. For the converse, note that these conditions become stronger as $m$ and $n$ increase, so it suffices to show that if they hold for $m=n=1$ then $\kdual T=c\kdual\Pi$. In this case we have $\kdual T\in\gen(P)\cap\cogen(I)=\Bild{c}$, so $\kdual T=ce\kdual T$. But by Proposition~\ref{rest-canonical-to-E}, we have $e\kdual T=\kdual\Pi$, and the result follows.
\end{proof}

As an application of this result, we obtain the promised result for the shifted modules and algebras of the algebra $\Gamma$ appearing in a Morita--Tachikawa triple.


\begin{thm}
\label{shifted-intexts}
Let $(A,E,\Gamma)$ be a Morita--Tachikawa triple, so $\domdim\Gamma=d+1$ for $d\geq1$, and write $\Pi$ for a maximal projective-injective summand of $\Gamma$. For each $0\leq k\leq d+1$, consider the shifted module $\shiftmod{k}$, its endomorphism algebra $\shiftalg{k}$, and let $c_k\colon\lmod{\Gamma}\to\lmod{\shiftalg{k}}$ be the intermediate extension functor from the recollement in \eqref{shifted-recollements}. Writing $P_k= \Hom_\Gamma(\shiftmod{k},\Pi)$ and $I_k= \kdual\Hom_\Gamma(\Pi,\shiftmod{k})$, we have
\[ 
\kdual\shiftmod{k} \in \gen_{d-k} (P_k ) \cap \cogen^{k-1} (I_k).
\]
If $1\leq k\leq d$, it then follows that
\[
\kdual\shiftmod{k}=c_kE.
\]
\end{thm}

\begin{proof}
Since $\domdim{\Gamma}=d+1$, we have
\begin{align*}
\Gamma\in\cogen^{d}(\Pi)&\subseteq\cogen^{k-1}(\Pi),\\
\kdual\Gamma\in\gen_{d}(\Pi)&\subseteq\gen_{d-k}(\Pi)
\end{align*}
for any $0\leq k\leq d+1$. To apply Proposition~\ref{CasIntExt}, it is therefore enough to check that 
\[
\Ext^1_{\Gamma}(\Omega^{-i} \Gamma ,\shiftmod{k}) = 0 = \Ext^1_{\Gamma }(\shiftmod{k}, \Omega^j \kdual \Gamma )
\]
for any $1\leq i\leq k$ and $1\leq j \leq d-k+1$, so fix $i$ and $j$ satisfying these constraints. Since $1\leq i,j\leq d+1$, the standard homological argument shows that
\[
\begin{aligned}
\Ext^n_{\Gamma}(\Omega^{-i} \Gamma, - ) &= \Ext^{n-i}_{\Gamma}(\Gamma, -)= 0,\\ 
\Ext^m_{\Gamma}(-, \Omega^j \kdual \Gamma ) &= \Ext^{m-j}_{\Gamma}(-, \kdual \Gamma )=0
\end{aligned}
\]
for all $n>i$ and $m>j$, using that the relevant syzygy and cosyzygy can be computed using projective-injective covers and envelopes. 
By the construction of $\shiftmod{k}$ from Proposition~\ref{shifting}, we then have
\begin{gather*}
\Ext^1_{\Gamma} (\Omega^{-i}\Gamma,\shiftmod{k}) = \Ext^{1+k}_{\Gamma} (\Omega^{-i}\Gamma , \Gamma )=0,\\
\Ext^1_{\Gamma}(\shiftmod{k}, \Omega^j\kdual \Gamma )  = \Ext^1_{\Gamma} (\Omega^{-k}\Gamma, \Omega^j \kdual \Gamma ) = \Ext^{2+d-k}_{\Gamma} (\Omega^{-(d+1)}\Gamma, \Omega^j \kdual \Gamma )=0
\end{gather*}
by the above calculations, noting that $1+k>i$ and $2+d-k>j$. Our desired conclusions now follow directly from Proposition~\ref{CasIntExt}.
\end{proof}

We close the section by stating the dual results for coshifted modules and algebras. Again in the setting of Proposition~\ref{rest-canonical-to-E}, write
\[ 
P'=\Hom_\Gamma(C,\Pi),\qquad
I'=\kdual\Hom_\Gamma(\Pi,C),\]
noting that $P'$ is projective, $I'=\nu P'$ is injective, and $P'\oplus I'\in\add{\kdual C}$. The dual of Proposition~\ref{CasIntExt}, obtained by swapping the roles of the two algebras, is as follows.

\begin{pro} \label{TasIntExt}
In the context of the preceding paragraph, let $m,n\geq0$.
\begin{itemize}
\item[(i)] The following are equivalent:
\begin{itemize} 
\item[(a)]
$\kdual\Gamma\in\gen_{m-1} (\Pi)$ and $\Ext^1_{\Gamma} (C,\Omega^i\kdual\Gamma)=0$ for $1\leq i \leq m$, and
\item[(b)] $\kdual C\in \gen_{m-1} (P^\prime )$. 
\end{itemize}
\item[(ii)] The following are equivalent:
\begin{itemize}
\item[(a)] $\Gamma\in \cogen^{n-1} (\Pi)$ and $\Ext^1_{\Gamma} (\Omega^{-i}\Gamma,C)=0$ for $1\leq i \leq n$, and
\item[(b)] $\kdual C\in \cogen^{n-1} (I^\prime)$.
\end{itemize}
\end{itemize}
Moreover the conditions in (i) and (ii) both hold for some $m,n\geq1$ if and only if $\kdual C$ is in the image of the intermediate extension functor $c'$ associated to $e'$, and in this case $\kdual C=c'(\kdual\Pi)$.
\end{pro}

This result can then be used to prove the following dual to Theorem~\ref{shifted-intexts}.

\begin{thm}
\label{coshifted-intexts}
Let $(A,E,\Gamma)$ be a Morita--Tachikawa triple, so $\domdim\Gamma=d+1$ for $d\geq1$, and write $\Pi$ for a maximal projective-injective summand of $\Gamma$. For each $0\leq k\leq d+1$, consider the coshifted module $\coshiftmod{k}$, its endomorphism algebra $\coshiftalg{k}$, and let $c^k\colon\lmod{\Gamma}\to\lmod{\coshiftalg{k}}$ be the intermediate extension functor from the recollement in \eqref{shifted-recollements}. Writing $P^k= \Hom_\Gamma(\coshiftmod{k},\Pi)$ and $I^k= \kdual\Hom_\Gamma(\Pi,\coshiftmod{k})$, we have
\[ 
\kdual \coshiftmod{k} \in \gen_{d-k} (P^k ) \cap \cogen^{k-1} (I^k).
\]
If $1\leq k\leq d$, it then follows that
\[
\kdual\coshiftmod{k}=c^kE.
\]
\end{thm}

\section{Examples}
\label{examples}

\begin{exa}
Let $A$ be the path algebra of a linearly-oriented quiver of type $\mathsf{A}_3$, and take $E$ basic with $\add{E}=\add{(A\oplus\kdual A)}$. Then $\Gamma=\End_A(E)^{\op}$ is isomorphic to the quotient of the path algebra of the quiver
\[\begin{tikzcd}[column sep=15pt]
1\arrow{r}&2\arrow{r}&3\arrow{r}&4\arrow{r}&5
\end{tikzcd}\]
by the ideal generated by all paths of length $3$, and has global dimension $3$. We have
\[\Pi=\Gamma(e_1+e_2+e_3)=\begin{smallmatrix}1\\&2\\&&3\end{smallmatrix}\oplus\begin{smallmatrix}2\\&3\\&&4\end{smallmatrix}\oplus\begin{smallmatrix}3\\&4\\&&5\end{smallmatrix},\]
and
\[T_1=\Pi\oplus\begin{smallmatrix}3\end{smallmatrix}\oplus\begin{smallmatrix}3\\&4\end{smallmatrix}\]
so we may compute $B_1$ to be the path algebra of
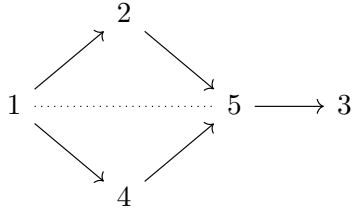
\[\begin{tikzcd}[]
&2\arrow{dr}\\
1\arrow{ur}\arrow{dr}\arrow[dotted,-]{rr}&&5\arrow{r}&3\\
&4\arrow{ur}
\end{tikzcd}\]
modulo the commutativity relation on the square. We see that $\gldim{B_1}=2$ (cf.~Proposition~\ref{gldim-bound}). We can compute that, as $B_1$-modules, we have
\begin{align*}
\ell_1(E)&=\begin{smallmatrix}&1\\\end{smallmatrix}\oplus\begin{smallmatrix}&1\\4&&2\\&5\end{smallmatrix}\oplus\begin{smallmatrix}&1\\4&&2\\&5\\&3\end{smallmatrix}\oplus\begin{smallmatrix}&2\\5\\3\end{smallmatrix}\oplus\begin{smallmatrix}3\end{smallmatrix},\\
r_1(E)&=\begin{smallmatrix}1\end{smallmatrix}\oplus\begin{smallmatrix}1\\&2\end{smallmatrix}\oplus\begin{smallmatrix}&1\\4&&2\\&5\\&3\end{smallmatrix}\oplus\begin{smallmatrix}4&&2\\&5\\&3\end{smallmatrix}\oplus\begin{smallmatrix}3\end{smallmatrix},
\end{align*}
so the image of the universal map is
\[c_1(E)=\begin{smallmatrix}1\end{smallmatrix}\oplus\begin{smallmatrix}1\\&2\end{smallmatrix}\oplus\begin{smallmatrix}&1\\4&&2\\&5\\&3\end{smallmatrix}\oplus\begin{smallmatrix}&2\\5\\3\end{smallmatrix}\oplus\begin{smallmatrix}3\end{smallmatrix}=\kdual\shiftmod{1},\]
as claimed in Theorem~\ref{shifted-intexts}.
\end{exa}

\begin{exa}
A simple but instructive family of examples is the following. Let $A$ be the path algebra of a linearly oriented $\mathsf{A}_n$ quiver modulo the ideal  generated by all paths of length $2$, and take $E$ basic with $\add{E}=\add(A\oplus\kdual A)$. Then $E$ is $(n-1)$-cluster-tilting, so $\Gamma=\End_A(E)^{\op}$ is an $(n-1)$-Auslander algebra, with dominant and global dimension $n$, and its families of shifted and coshifted algebras coincide, with $\shiftalg{k}\cong\coshiftalg{n-k}$, by Theorem~\ref{tilt-cotilt-dAG}.

We compute that the $k$-th coshifted algebra $\coshiftalg{k}$ may be presented as the path algebra of the linearly oriented $\mathsf{A}_{n+1}$ quiver
\[\begin{tikzcd}[column sep=15pt]
1\arrow{r}&2\arrow{r}&\cdots\arrow{r}&n+1
\end{tikzcd}
\]
modulo all paths of length $2$ except that with middle vertex $k+1$. It follows that
\[\gldim{\coshiftalg{k}}=\max\{n-k,k\}.\]
\end{exa}

\begin{exa} It can happen that the dominant dimension of a shifted algebra is again positive, allowing us to iterate sequences of shifts and coshifts. We illustrate this on the Auslander algebra
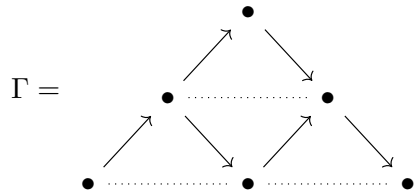
\[\Gamma=\mathord{\begin{tikzcd}[column sep=15pt]
&&\bullet\arrow{dr}\\
&\bullet\arrow{ur}\arrow[dash,dotted]{rr}\arrow{dr}&&\bullet\arrow{dr}\\
\bullet\arrow{ur}\arrow[dash,dotted]{rr}&&\bullet\arrow{ur}\arrow[dash,dotted]{rr}&&\bullet
\end{tikzcd}}\]
of the path algebra of a linearly oriented $\mathsf{A}_3$ quiver. We may compute that the first shifted algebra is
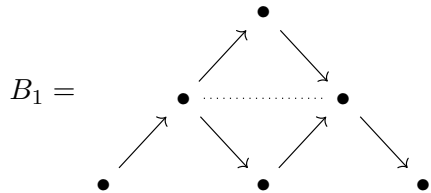
\[\shiftalg{1}=\mathord{\begin{tikzcd}[column sep=15pt]
&&\bullet\arrow{dr}\\
&\bullet\arrow{ur}\arrow[dash,dotted]{rr}\arrow{dr}&&\bullet\arrow{dr}\\
\bullet\arrow{ur}&&\bullet\arrow{ur}&&\bullet
\end{tikzcd}}\]
(noting the absence of relations in the lowest row) and then use Theorem~\ref{dAG-algebras} to see that $\coshiftalg{1}\cong\shiftalg{1}$, $\shiftalg{2}\cong\coshiftalg{0}\cong\Gamma$ and $\coshiftalg{2}\cong\shiftalg{0}\cong\Gamma$.
Since $\domdim B_1 =1 $, we can shift again to obtain
\[\shiftalg{1,1}=\mathord{\begin{tikzcd}
\bullet\arrow{r}&\bullet\arrow{r}&\bullet\arrow{r}&\bullet\\
\bullet\arrow{u}\arrow{r}\arrow[dash,dotted]{ur}&\bullet\arrow{u}
\end{tikzcd}}\]
which also has dominant dimension $1$; note in particular that $\shiftalg{1,1}\not\cong\shiftalg{2}$. Shifting once more, we find
\[\shiftalg{1,1,1}=\mathord{\begin{tikzcd}
\bullet\arrow{r}&\bullet\arrow{r}&\bullet\arrow{r}&\bullet\arrow{r}&\bullet\\
&\bullet\arrow{u}\end{tikzcd}}\]
which has dominant dimension $0$, so the sequence ends.
\end{exa}

\section*{Acknowledgements}
We would like to thank William Crawley-Boevey, Sondre Kvamme, Greg Stevenson and the anonymous referees for helpful comments. This project began while the first author visited Universit\"at Bielefeld with funding from the Deutsche Forschungsgemeinschaft grant SFB 701---he thanks Henning Krause and the rest of the Bielefeld representation theory group for their hospitality. Subsequently, he was supported by a postdoctoral fellowship from the Max-Planck-Gesellschaft.
The second author is supported by the Alexander von Humboldt-Stiftung in the framework of an Alexander von Humboldt Professorship endowed by the Federal Ministry of Education and Research.

\bibliographystyle{amsalpha}
\bibliography{ShiftedModules-MT}

\end{document}